\documentclass[11pt, a4paper]{article}
\usepackage{amsmath,amssymb,amsthm}
\usepackage{mathtools}
\usepackage{amsfonts}
\usepackage{physics}
\usepackage{faktor}

\usepackage[T1]{fontenc}
\usepackage[utf8]{inputenc}
\usepackage[english]{babel}
\usepackage{kpfonts}
\usepackage{microtype}
\usepackage[autostyle]{csquotes}

\usepackage[style=alphabetic,url=false, doi=false, date=year, backend=bibtex,maxbibnames=10, firstinits=true]{biblatex}
\AtEveryBibitem{\clearfield{issn}}
\AtEveryCitekey{\clearfield{issn}}
\AtEveryBibitem{\clearfield{note}}
\AtEveryCitekey{\clearfield{note}}
\usepackage{enumitem}
\usepackage{hyperref}

\usepackage{tabularx}

\usepackage[algoruled]{algorithm2e}
\usepackage{cleveref}
\usepackage{xcolor}
\usepackage{graphicx}

\usepackage[margin=2.25cm]{geometry}
\usepackage{comment}

\theoremstyle{plain}
\newtheorem{theoremintro}{Theorem}

\newtheorem{theorem}{Theorem}[section]
\newtheorem{lemma}[theorem]{Lemma}
\newtheorem{prop}[theorem]{Proposition}
\newtheorem{corollary}[theorem]{Corollary}

\theoremstyle{definition}
\newtheorem{definition}[theorem]{Definition}
\newtheorem{example}[theorem]{Example}

\theoremstyle{remark}
\newtheorem*{remark}{Remark}

\DeclarePairedDelimiter{\ceil}{\lceil}{\rceil}
\DeclarePairedDelimiter{\floor}{\lfloor}{\rfloor}

\newcommand{\N}{\mathbb N}

\newcommand{\C}{\mathbb C}

\newcommand{\R}{\mathbb R}

\renewcommand{\epsilon}{\varepsilon}

\newcommand{\cC}{\mathcal{C}}
\newcommand{\cI}{\mathcal{I}}

\newcommand{\cS}{\mathcal{S}}
\newcommand{\cO}{\mathcal{T}}
\newcommand{\cQ}{\mathcal{Q}}

\newcommand{\cL}{\mathcal{L}}
\newcommand{\cLone}{\mathcal{L}^{(1)}}
\newcommand{\cV}{\mathcal{V}}
\renewcommand{\cS}{\mathcal{S}}
\newcommand{\cM}{\mathcal{M}}
\newcommand{\cMone}{\mathcal{M}^{(1)}}

\newcommand*{\tigen}[2]{\langle {{#1}} \rangle_{{#2}}}
\newcommand*{\tqgen}[2]{\cQ_{{#2}}({\vb{#1}})}
\newcommand*{\ord}[1]{\Pi{#1}}

\newcommand{\intdeg}{\theta}

\DeclareMathOperator{\hess}{Hess}
\DeclareMathOperator{\cP}{Pos}
\DeclareMathOperator{\pos}{Pos}

\DeclareMathOperator{\cone}{cone}
\DeclareMathOperator{\conv}{conv}
\DeclareMathOperator{\ann}{Ann}

\DeclareMathOperator{\supp}{supp}
\def \sos {\mathrm{SoS}}
\def \mom {\mathrm{MoM}}

\def \eval {\mathbf{e}}

\setlist[enumerate,1]{label={(\roman*)},ref={\thetheorem (\roman*)}}
\renewcommand{\braket}[2]{\langle #1 , #2 \rangle}
\renewbibmacro{in:}{}
\begin{document}
\title{Exact Moment Representation in Polynomial Optimization}
\author{Lorenzo Baldi, Bernard Mourrain\\
  \ \\
{Inria d'Universit\'e C\^ote d'Azur, Sophia Antipolis, France}
}
\date{}

\maketitle
\begin{abstract}
  We investigate the problem of representing moment sequences by measures in the context of Polynomial Optimization Problems, that consist in finding  the infimum of a real polynomial on a real semialgebraic set defined by polynomial inequalities. 
  We analyze the exactness of Moment Matrix (MoM) hierarchies,
  dual to the Sum of Squares (SoS) hierarchies, which are sequences of convex cones introduced by Lasserre to approximate measures and positive polynomials.
  We investigate in particular flat truncation properties, which 
  allow testing effectively when MoM exactness holds and recovering the minimizers.
  
  We show that the dual of the MoM hierarchy coincides with 
  the SoS hierarchy extended with the real radical of the support of the defining quadratic module $Q$. 
  We deduce 
  that flat truncation happens if and only if the support of the quadratic module associated with the minimizers is of dimension zero. We also bound the order of the hierarchy at which flat truncation holds.
  
  As corollaries, we show that flat truncation and MoM exactness hold
when regularity conditions, known as Boundary Hessian Conditions, hold (and thus that MoM exactness holds generically);
and when the support of the quadratic module $Q$ is zero-dimensional.
Effective numerical computations illustrate these flat truncation properties.
\end{abstract}

\section{Introduction}

Let $f, g_1,\dots,g_s \in \R[X_1,\dots,X_n]$ be polynomials in the indeterminates $X_{1}, \ldots, X_{n}$ with real coefficients. The goal of Polynomial Optimization is to find:
\begin{equation}\label{eq:pop}
      f^* \coloneqq \inf \ \big\{ \, f(x)\in \R \mid x \in \R^n, \ g_i(x) \ge 0 \ \textup{ for } i=1, \ldots,s \,\big\},
    \end{equation}
    that is the infimum $f^{*}$
of the \emph{objective function} $f$ on the \emph{basic, {closed} semialgebraic set} $S \coloneqq \{ \, x \in \R^n \mid \ g_i(x) \ge 0 \ \textup{ for } i=1, \ldots,s \, \} $.
This is a general problem, which appears in many contexts and with many applications. To cite a few of them: in graph theory \cite{https://doi.org/10.48550/arxiv.2103.01574}, network optimization design \cite{6980142}, control \cite{6606873}: see \cite{lasserre_moments_2010} for a more comprehensive list. {Equality constraints are also allowed in this setting, since $g(x) = 0$ if and only if $g(x)\ge 0$ and $-g(x) \ge 0$. We can then consider optimization problems on real algebraic varieties, i.e. common zero loci of finitely many real polynomials.}
{Moreover, many famous NP-hard and NP-complete problems can be rephrased as polynomial optimization problems, see e.g. \cite[Sec.~1.1]{laurent_sums_2009}}.

To (approximately) solve such problems, Lasserre \cite{lasserre_global_2001} proposed to use two hierarchies of finite dimensional convex cones depending on an order $d\in \N$, and he proved, for {\emph{Archimedean quadratic modules}}, the convergence of the optima associated to these hierarchies to the minimum $f^{*}$ of $f$ on $S$, when $d\rightarrow \infty$.
The first hierarchy replaces non-negative polynomials by Sums of Squares (SoS) and non-negative polynomials on $S$ by polynomials of degree $\le d$ in the truncated quadratic module $\cQ_{2d}(\vb g)$ generated by the tuple of polynomials $\vb g=\{ \, g_{1},\ldots, g_{s} \, \}$.

The second and dual hierarchy replaces positive measures by linear functionals $\cL_{2d}(\vb g)$ which are non-negative on the polynomials of the truncated quadratic module $\cQ_{2d}(\vb g)$.
We will describe {these constructions} more precisely in section \ref{subsec::poly_opt}.

This approach has many interesting properties (see e.g. \cite{lasserre_introduction_2015, laurent_sums_2009, marshall_positive_2008}).
It was proposed with the aim to recover the infimum $f^{*}$, and if this infimum is reached, the set of minimizers $S^{\min}:=\{\xi \in S \mid f(\xi)=f^{*}\}$. The extraction of minimizers is strongly connected to the so called \emph{flat truncation} property, that will be the focus of the paper.

To tackle these challenges, one can first address the finite convergence problem, that is when the value $f^{*}$ can be obtained at a given order of the relaxation(s). The second problem is the exactness of the hierarchies: we call the Sum of Squares (SoS) hierarchy exact when the non-negative polynomial $f-f^{*}$ belongs to the truncated quadratic module $\cQ_{2d}(\vb g)$ for some $d\in \N$; and we call the Moment Matrix (MoM) hierarchy exact when, for {all $d\in \N$ big enough} , any optimal linear functional $\lambda^{*}\in \cL_{2d}(\vb g)$ is {represented by} a positive measure supported on $S$. We are going to investigate in detail this {exactness property of the MoM hierarchy}.

Several works have been developed over the last decades to address SoS representation problems.
\cite{parrilo_explicit_2002} {observed} that if the complex variety $\cV_{\C}(I)$ defined by an ideal $I = (\vb h)$ generated by real polynomials is finite and $I$ is radical, then $f-f^{*}$ has a representation as a sum of squares modulo $I$.
\cite{laurent_semidefinite_2007} showed the finite convergence property if the complex variety $\cV_{\C}(I)$ is finite, and {that} the truncated moment sequence {at some level of the hierarchy} has a representing measure, if moreover the ideal $I$ is radical.
\cite{nie_polynomial_2013} showed that if the semialgebraic set $S$ is finite, then the finite convergence property holds for a preordering defining $S$.

\cite{scheiderer_distinguished_2005}
proved that $f-f^{*}$ is in the quadratic module $\cQ$ defining $S$ modulo {the ideal} $(f-f^{*})^{2}$ if and only if $f-f^{*} \in \cQ$, and then the SoS hierarchy is exact.
\cite{MarshallRepresentationsnonnegativepolynomials2006}, \cite{MarshallRepresentationsNonNegativePolynomials2009}
proved that, under regularity conditions at the minimizers, known as {\emph{Boundary Hessian Conditions}} (BHC), $f-f^{*}$ is in the quadratic module, {i.e.,} the SoS exactness property holds.
\cite{nie_minimizing_2006}, \cite{demmel_representations_2007} showed that, by adding gradient constraints when $S=\R^{n}$ or KKT constraints when $S$ is a general basic semialgebraic set,
the SoS exactness property holds when the corresponding Jacobian ideal is radical.
\cite{nie_exact_2013} showed that, by adding the Jacobian constraints,
the finite convergence property holds under some regularity assumption.
\cite{nie_certifying_2013} showed that finite convergence and the flat truncation property are equivalent under generic assumptions, if the SoS hierarchy is exact and strong duality holds.
In \cite{nie_optimality_2014}, it is shown that BHC imply finite convergence and that BHC are generic.
\cite{kriel_exactness_2019} showed the SoS exactness property if the quadratic module defining $S$ is Archimedian and some strict concavity properties of $f$ at the finite minimizers are satisfied.

Though many works focused on the SoS hierarchy and the representation of positive polynomials
with Sums of Squares, the MoM hierarchy has been much less studied. We mention \cite{lasserre_semidefinite_2008} and \cite{lasserre_moment_2013}, which prove that if $S$ is finite, the value $f^{*}$, the minimizers and the vanishing ideal of $S$ can be recovered from moment matrices associated with the truncated preordering defining $S$.

{From a methodological and practical point of view, flat truncation tests on moment matrices, see e.g. \cite{CurtoFlatExtensionsPositive1998} and \cite{Laurentgeneralizedflatextension2009}, are a way to decide finite convergence, i.e. whether the minimum $f^*$ is reached at some order of the hierarchy (another approach is the comparison of the lower bound obtained with an objective value at a local minimizer).}
But flat truncation also implies MoM exactness. 
Moreover, it allows extracting the 
finite minimizers from moment matrices
(see \cite{Henrion05detectingglobal}, \cite{mourrain_polynomialexponential_2018}), whereas SoS exactness does not yield the minimizers.
Therefore a natural question, of theoretical and practical importance, is:
\begin{quote}\it
When does flat truncation hold in a Polynomial Optimization Problem? 
\end{quote}

It is known that truncated minimizing {pseudo-moment} sequences are not always {represented by} measures, see \Cref{appendix}, and thus flat truncation does not hold in general.
{But surprisingly, no {algebro}-geometric characterization of when flat truncation holds has been described in the last decades.}

\textbf{Contributions}. { Our main contribution is a complete characterization of flat truncation in terms of the zero-dimensionality of a natural ideal defining the minimizers of the Polynomial Optimization Problem. Furthermore, we show that flat truncation for \emph{generic} minimizing linear functionals implies the flat truncation property for all minimizing linear functionals, and the exactness property. 
\begin{theoremintro}[{see \Cref{thm::flat_implies_exact,thm::flat_iff}}]\label{thm:A}
    Assume that we have MoM finite convergence. Then $\dim \frac{\R[\vb X]}{\supp (\cQ(\vb g) + (f-f^*))} = 0$ if and only if there exists $d$ such that a generic $\lambda^* \in \cL_{2d}^{\min}(f; \vb g)$ has flat truncation. 
    
    In this case, all $\lambda \in \cL_{2d}^{\min}(f; \vb g)$ have flat truncation, and the MoM hierarchy is exact.
\end{theoremintro}
Above, $\supp (\cQ(\vb g) + (f-f^*)) := (\cQ(\vb g) + (f-f^*)) \cap - (\cQ(\vb g) + (f-f^*))$ is an ideal, associated with the quadratic module $\cQ(\vb g) + (f-f^*)$ which defines the minimizers $S^{\min}$; $\cL_{2d}^{\min}(f; \vb g)$ is the face of the minimizers of the MoM relaxation of order $d$; and linear functionals in the relative interior of $\cL_{2d}^{\min}(f; \vb g)$ are generic (see \Cref{def::generic_t} for a precise definition). 

\Cref{thm:A} easily allows to conclude that flat truncation occurs when regularity conditions, known as \emph{Boundary Hessian Conditions} (BHC) hold true, see \Cref{def:BHC}. This conditions hold for generic $\vb g$ and $f$. 

\begin{theoremintro}[{see \Cref{thm::BHC} and \Cref{cor::generic_exactness}}]\label{thm:B}
    Let $f \in \R[\vb{X}]$, $Q=\cQ(\vb g)$ be an Archimedean finitely
    generated quadratic module and assume that the BHC hold at every
    minimizer of $f$ on $S=\cS(\vb g)$.
    Then the SoS hierarchy is exact, the MoM hierarchy is exact, and the flat truncation holds for {all} $\lambda \in \cL_{2d}^{\min}(f;\vb g)$ when $d$ is big enough.
    Moreover, this condition occurs for generic $f$ and $\vb g$ satisfying the Archimedean condition.
\end{theoremintro}

\Cref{thm:B} extends the results on finite convergence, SoS exactness and flat truncation proved in \cite{MarshallRepresentationsnonnegativepolynomials2006}, \cite{MarshallRepresentationsNonNegativePolynomials2009},
\cite{nie_certifying_2013} and \cite{nie_optimality_2014}. A detailed comparison with these references is technical and therefore developed to \Cref{sec::flat_tru}.

Another consequence of \Cref{thm::flat_iff}, shown in  \Cref{th::finvar_exact_qua}, is that when the set $S$ is finite, flat truncation holds if the quotient by the support of the quadratic module $\cQ$ is of dimension zero.
This generalizes results of \cite{lasserre_semidefinite_2008} and \cite{lasserre_moment_2013} on semidefinite moment representations on finite sets.

To prove these results, we investigate in detail the properties of truncated moment relaxations and their duals, i.e. properties of truncated quadratic modules $\tqgen{g}{d}$ and positive linear functionals $\cL_{d}(\vb g)$. These properties are summarized in \Cref{thm:C}.
\begin{theoremintro}[{see \Cref{thm::generic_t,prop:flat truncation,thm::fin_set_qua}}]\label{thm:C}
    If $d$ is big enough, then the kernel of the (truncated) moment matrix of a generic $\lambda^* \in \cL_{d}(\vb g)$ generates the real radical of $\supp \cQ(\vb g)$.
    Furthermore, if $\lambda^*$ has flat truncation, then $\supp \cQ(\vb g)$ is zero-dimensional, $\cS(\vb g)$ is finite and the flat truncation degree is at least the interpolation degree of $\cS(\vb g)$.
    
    Conversely, if $\supp \cQ(\vb g)$ is zero-dimensional and $d$ is big enough, then all $\lambda \in \cL_d(\vb g)$ have the flat truncation property.
\end{theoremintro}

When the quotient by the support of $Q$ is of dimension zero, \Cref{thm:C} shows that (truncated) linear functionals in $\cL_{d}(\vb g)$ {are all represented by} measures supported on $S$, or in other words they are {represented by convex combinations of} evaluations at the points of $S$. This is therefore a particular solution of the moment problem, in the spirit of Curto-Fialkow's flat extension \cite{CurtoFlatExtensionsPositive1998}. A detailed comparison with related results in the zero-dimensional case, such as \cite{lasserre_semidefinite_2008} and \cite{lasserre_moment_2013}, is performed in \Cref{sec:Geometry}. The characterization of the ideal generated by the kernel of the moment matrix for generic linear functionals is new in the positive dimensional case.
}

\textbf{Outline}. The paper is structured as follows. In the next section of the introduction, we define the algebraic objects that we will use and recall their main properties. In \Cref{sec:examples}, we
describe in detail the notions of finite convergence and exactness for the Sum of Squares (SoS) and Moment Matrix (MoM) hierarchies. We present several examples showing how these notions are related (these examples are detailed in \Cref{appendix}).

In \Cref{sec:Geometry}, we investigate the properties of truncated moment sequences (\Cref{sec:truncatedmomrep}), 
of their annihilators (\Cref{ref:annihilator})
and we analyze when flat truncation holds and relate it with the interpolation degree (\Cref{subsec::interpolation_bases}).

In \Cref{sec::flat_tru}, we apply these results to Polynomial Optimization Problems (POPs).
In \Cref{sec:flat truncation degree}, we prove necessary and sufficients conditions for flat truncation and analyze at which degree flat truncation holds and yields the minimizers. 
We prove that exactness and flat truncation hold for POPs satisfying the Boundary Hessian Conditions (\Cref{sec::BHC}), {and under a zero-dimensionality assumption, which implies that the underlying semialgebraic set is finite (\Cref{sec:finite set})}.

For the numerical computations performed on the examples, which illustrate these developments, we use the Julia package
\texttt{{MomentPolynomialOpt.jl}}\footnote{https://github.com/AlgebraicGeometricModeling/MomentPolynomialOpt.jl}
with the SDP solvers \texttt{Mosek} and \texttt{SDPA}, based on
interior point methods.

\subsection{Notation}
We {recall} here basic notation and definitions we need hereafter, {which can be found e.g. in the textbooks \cite{Cox2015,lasserre_moments_2010,marshall_positive_2008}}.

 If $A$ is a subset of a $\R$-vector space $V$, we denote $\cone(A)$ the convex cone generated by $A$, by $\conv(A)$ its convex hull and by $\langle A \rangle$ its linear span. 

\vspace{-0.3cm}\paragraph{Polynomials.} Let $\R[\vb{X}] \coloneqq \R[X_1,\dots,X_n]$ be the $\R$-algebra of polynomials in $n$ indeterminates $X_{1}, \ldots, X_{n}$. 
We denote $(h_1,\dots,h_r)\subset \R[\vb{X}]$ the \emph{ideal} generated by $h_1,\dots,h_r\in \R[\vb{X}]$. 

If $A \subset \R[\vb{X}]$ and $t\in \N$, $A_t\coloneqq \{ \,f \in A \mid \deg f \le t \, \}$. In particular $\R[\vb{X}]_t$ is the vector space of polynomials of degree $\le t$.

Given a finite {set} of polynomials $\vb g = \{ \, g_1,\dots,g_r \, \}$, we define $\ord{\vb g} \coloneqq \{ \prod_{j \in J} g_j \colon {\emptyset \neq} J \subset \{ 1,\dots , r \} \} = \{ g_1, \dots, g_r,\, g_1 g_2, g_1 g_3,  \dots, g_1 \cdots g_r\} $, the set of all the products of the $g_i$'s, and $\pm \vb g \coloneqq \{ g_1, -g_1, \dots, g_r, -g_r \}$. 

If $A \subset \R[\vb{X}]$ we define $\cS(A)\coloneqq \big\{ \, x \in \R^n \mid f(x) \ge 0 \ \forall f \in A \, \big\}$. In particular we denote $\cS(\vb g) = \big\{ \, x \in \R^n \mid g(x) \ge 0 \ \forall g \in \vb g \, \big\}$ the \emph{basic, {closed} semialgebraic set defined by $\vb g$}. If $S \subset \R^n$,
we denote $\pos(S)=\{f \in \R[\vb X] \colon {f(x)\ge 0 \ \forall x\in S} \}$ the {convex} cone of nonnegative polynomials on $S$.

Let $\Sigma^2=\Sigma^2 [\vb X] \coloneqq \big\{ \, f \in \R[\vb{X}] \mid \exists r \in \N, \ h_i \in \R[\vb{X}] \colon f = h_1^2 + \dots + h_r^2   \,\big\}$ be the convex cone of \emph{Sum of Squares polynomials} (SoS).
$Q \subset \R[\vb{X}]$ is called \emph{quadratic module} if $1\in Q$, $\Sigma^2\cdot Q\subset Q$ and $Q+Q\subset Q$. If in addition $Q \cdot Q \subset Q$, $Q$ is a \emph{preordering}. For a {quadratic module} $Q \subset \R[\vb X]$, we define $\supp Q \coloneqq Q \cap -Q$. $\supp Q$ is an ideal of $\R[\vb X]$, see e.g. {\cite[Prop.~2.1.2]{marshall_positive_2008}}.
Given a finite set $\vb g = \{ \, g_1, \dots g_r \, \}$, we denote $\cQ(\vb g)$
the smallest quadratic module containing $\vb g$, i.e. $\cQ(\vb g) = \Sigma^2 + \Sigma^2 \cdot g_1 + \dots +\Sigma^2 \cdot g_r$. {Quadratic modules of the form $\cQ(\vb g)$ for finite $\vb g \subset \R[\vb X]$ are called \emph{finitely generated}}. In a similar way, we denote $\cO(\vb g) = \cQ(\Pi \vb g)$ the smallest preordering containing $\vb g$, also called finitely generated if $\vb g$ is finite.

We say that a quadratic module $Q$ is \emph{Archimedean} if $\exists \ 0 \le r \in \R \colon  r-\norm{\vb X }^2 \in Q$.
Notice that if $Q$ is a {finitely generated} Archimedean quadrartic module then $\cS(Q)$ is compact, and these conditions are equivalent for {finitely generated} preorderings {\cite{wörmann1998strikt,Berr2001}}. {We also recall that a finitely generated quadratic module $Q$ is Archimedean if and only if there exists $g \in Q$ such that $\cS(g)$ is compact, see e.g. \cite[Th.~7.1.1]{marshall_positive_2008}.}%

{For an ideal $I \subset \R[\vb X]$,} the \emph{real radical} of $I$, denoted $\sqrt[\R]{I}$, is the ideal:
    \[
        \sqrt[\R]{I} \coloneqq \{ \, f \in \R[\vb X] \mid \exists m \in \N, \ s \in \Sigma^2 \text{ with } f^{2m} + s \in I \, \}.
    \]
    We say that $I$ is \emph{real} or \emph{real radical} if $I = \sqrt[\R]{I}$.

Recall that $\sqrt[\R]{I} = \sqrt{\supp (I+\Sigma^2)}$, see e.g. {\cite[p.~23]{marshall_positive_2008}},
and thus $\sqrt[\R]{I}$ is a radical ideal. We are in particular interested in the case $I = \supp Q = Q \cap -Q$ for an arbitrary quadratic module $Q \subset \R[\vb X]$.
In this case, the real radical and radical coincide: $\sqrt[\R]{\supp Q}= \sqrt{\supp Q}$, see e.g. \cite[Note 2.1.4]{marshall_positive_2008}.

\vspace{-3mm}\paragraph{Linear functionals, truncations and moment matrices.}
For a $\R$-vector space $V$, $V^*$ is the dual space of linear functionals on $V$.
For $\lambda \in V^*$,  we denote $\braket{\lambda}{v} = \lambda (v)$ the application of $\lambda$ to $v \in V$.
For $A \subset V$, 
we define $A^{\perp} \coloneqq \big\{\, \lambda \in  V^* \mid \braket{\lambda}{a}=0 \ \forall a \in A \, \big\}$
and $A^\vee \coloneqq \big\{\, \lambda \in  V^* \mid \braket{\lambda}{a}\ge 0 \ \forall a \in A \, \big\}$.

Let $(\R[\vb{X}])^*$ be the vector space of linear functionals on $\R[\vb X]$. Recall that $(\R[\vb{X}])^*\cong \R[[\vb{Y}]] \coloneqq \R[[Y_1,\dots,Y_n]]$, with the isomorphism given by:
$
    (\R[\vb{X}])^* \ni \lambda \mapsto \sum_{\alpha \in \N^n} \braket{\lambda}{\vb X^{\alpha}} {\vb Y^{\alpha}} \in \R[[\vb{Y}]],
$
where $\{{\vb Y^{\alpha}}\}$ is dual to $\{\vb X^{\alpha} \}$, i.e. $\braket{\vb Y^{\alpha}}{\vb X^{\beta}}=\delta_{\alpha,\beta}$. 
When necessary, we will identify $\lambda \in (\R[\vb{X}])^*$ with its sequence of coefficients  $(\lambda_{\alpha})_{\alpha}$ called \emph{pseudo-moments}, in analogy to the case of a measure, where $\lambda_{\alpha}\coloneqq \braket{\lambda}{\vb X^{\alpha}}$.
{See \cite{mourrain_polynomialexponential_2018} for more details on this approach, or \cite[\S 3.2.1]{lasserre_moments_2010} for a classical presentation based on pseudo-moment sequences. 
Using linear functionals instead of pseudo-moment sequences allows us to treat dual elements, independently of any choice of basis in the primal space.}

If $t\le s\in \N $ and $\lambda \in (\R[\vb{X}]_s)^*$ (or $\lambda \in \R[\vb{X}]^*$), then $\lambda^{[t]} \in (\R[\vb{X}]_t)^*$ denotes its restriction to $\R[\vb X]_t$. Similarly if $B\subset (\R[\vb{X}]_s)^*$ then $B^{[t]} \coloneqq \{ \, \lambda^{[t]} \in (\R[\vb{X}]_t)^* \mid \lambda \in B \, \}$.

For $t\le r \in \N$, $\lambda \in (\R[\vb{X}]_r)^*$ and $g \in \R[\vb{X}]_t$, we define the \emph{convolution of $g$ and $\lambda$} as the element of $(\R[\vb{X}]_{r-t})^*$ defined by $g \star \lambda : f \in \R[\vb X]_{r-t} \mapsto \braket{g \star \lambda}{f} = \braket{\lambda}{gf}$.
We denote $\ann_t(\lambda)$ the \emph{annihilator} of $\lambda$ w.r.t. $\star$ in degree $\le t$, that is $\ann_t(\lambda) = \{ p\in \R[\vb X]_t\mid p \star \lambda=0\}$. 
Given $\lambda \in (\R[\vb{X}]_r)^*$, $r \ge 2t$, we define the Hankel operator:
\begin{align*}
    H_{\lambda}^t \colon \R[\vb X]_t & \to (\R[\vb X]_t)^* \\
    p & \mapsto (p \star \lambda)^{[t]}.
\end{align*}

The \emph{moment matrix} of $\lambda$ in degree $t$ is the matrix $H^t_{\lambda}=(\lambda_{\alpha+\beta})_{|\alpha|\le t,|\beta|\le t}$ of the Hankel operator $H^t_{\lambda}$ with respect to the bases $\{\vb X^{\beta} \}$ and  $\{{\vb Y^{\alpha}}\}$. 
Notice that the moment matrix $H^k_{\lambda}$ can be also identified with the symmetric operator associated to the quadratic form $p\in \R[\vb X]_k \mapsto \braket{\lambda}{p^2}$. 
By definition, the kernel of the moment matrix $H_{\lambda}^t$ is the {annihilator} of $\lambda$ in degree $\le t$: $\ann_t(\lambda)= \ker H^t_{\lambda}$.

If $s \le t$, we can identify the matrix of $H^s_{\lambda}$ with the submatrix of $H^t_{\lambda}$ indexed by monomials of degree $\le s$.
The \emph{localizing matrix} of $\lambda$ at $g\in \R[\vb X]$ is the matrix $H^t_{g \star \lambda}=((g \star \lambda)_{\alpha+\beta})_{\alpha,\beta} = (\sum_{\gamma}g_{\gamma}\lambda_{\alpha+\beta+\gamma})_{\alpha,\beta}$ of the Hankel operator $H^t_{g \star \lambda}$. This coincides with the definition of localizing matrix used in the literature, see e.g. \cite[Eq. (3.14)]{lasserre_moments_2010}.

\section{Finite Convergence and Exactness}\label{sec:examples}
We describe now Lasserre SoS and MoM relaxations \cite{lasserre_global_2001}, and we define the \emph{exactness} property. Hereafter we assume that the minimum $f^*$ of the objective function $f$ is always attained on $S$, that is: $S^{\min}\coloneqq\{\, x \in S \mid f(x)=f^*\,\}\neq \emptyset$.
\subsection{Polynomial optimization relaxations}\label{subsec::poly_opt}

The construction of polynomial optimization relaxations relies on the approximation of the cone of positive polynomials by tractable convex cones, that we recall hereafter. 

\noindent\textbf{Lasserre's SoS hierarchy.} 
For $d\in \N$ let $\Sigma^2_d= \Sigma^2\cap \R[\vb X]_d$ be the finite dimensional convex cone of SoS of degree $\le d$.
For $d\in \N$ and $\vb g = \{ \, g_1, \dots g_r \, \}\subset \R[\vb X]$, let 
$$\cQ_d(\vb g) \coloneqq \big\{ \, s_0 + \sum_{j=1}^r s_j g_j \in \R[\vb{X}]_d \mid r \in \N, \ s_0 \in \Sigma^2_d, \ s_j \in \Sigma^2_{d-\deg g_j} \, \big\}
$$ be the \emph{truncated quadratic module} generated by $\vb g$. {For the applications in polynomial optimization, we are interested in quadratic modules truncated at even degrees.}

We define the \emph{SoS relaxation of order $d$} of problem \eqref{eq:pop} as $\tqgen{g}{2d}$ and the supremum:
\begin{equation}
    \label{def::sosrel}
  f^*_{\sos,d}  \coloneqq \sup \big\{ \, a \in \R \mid f-a  \in \tqgen{g}{2d} \,\big\}.
\end{equation}

\noindent\textbf{Lasserre's MoM hierarchy.} To define the dual approximation of the polynomial optimization problem, we define
for $d\in \N$ and $\vb g = \{ \, g_1, \dots g_r \, \} \subset \R[\vb X]_d$:
\[
    \cL_d(\vb g) \coloneqq \cQ_d(\vb g)^\vee = \{ \, \lambda \in (\R[\vb X]_d)^* \mid  \braket{\lambda}{q}\ge 0 \ \forall q \in \cQ_d(\vb g) \, \}
\]
the cone of \emph{positive linear functionals} on $\cQ_d(\vb g)$,
which is the \emph{dual convex cone} to $\tqgen{g}{d}$, see \cite[Sec.~3.6]{marshall_positive_2008}. 
By conic duality, $\overline{\tqgen{g}{d}}=\cL_d(\vb g)^{\vee}$.

We have $\lambda \in \cL_d(\vb g)$ if and only if $\braket{\lambda}{s}\ge 0 \ \forall s \in \Sigma_d^2$ and $\braket{\lambda}{sf}\ge 0 \ \forall g \in \vb g, \forall s \in \Sigma^2$ such that $\deg gs \le d$. Another equivalent way to describe 
$\cL_d(\vb g)$ is using positive semidefinite matrices or \emph{Linear Matrix Inequalities}, since  the $\cL_d(\vb g)$ are spectrahedra: $\lambda \in \cL_d(\vb g)$ if and only if the symmetric matrices  $H_{\lambda}^{\floor{\frac{d}{2}}}$, $H_{g_1 \star\lambda}^{\floor{\frac{d - \deg g_1}{2}}}$, \dots, $H_{g_r \star \lambda}^{\floor{\frac{d- \deg g_r}{2}}}$ are positive semidefinite. {As before, for the applications in polynomial optimization we are interested in positive linear functionals acting on even degree polynomials.}

We consider an affine hyperplane section of the cone $\cL_{2d}(\vb g)$:
\[
    \cLone_{2d}(\vb g) \coloneqq \big\{ \, \lambda \in \cL_{2d}(\vb g) \mid \braket{\lambda}{1}=1 \, \big\}.
\]
This will be the set of feasible pseudo-moment sequences of the MoM relaxation of order $d$. 
With this notation we define the \emph{MoM relaxation of order $d$} of problem \eqref{eq:pop} as: 
\begin{equation}
\label{def::momrel}
  f^*_{\mom,d}  \coloneqq \inf \big\{ \, \braket{\lambda}{f} \in \R \mid \lambda \in \cLone_{2d}(\vb g) \,\big\}.
\end{equation}
We easily verify that
$f^*_{\sos,d} \le f^*_{\mom,d} \le f^{*}$.

\begin{definition}
    Consider the problem of minimizing $f \in \R[\vb{X}]$ on $\cS(\vb g)$. We define the set of \emph{minimizing linear functionals} at relaxation order $d$ as the linear functionals $\lambda$ minimizing \eqref{def::momrel}, i.e.

    \[
        \cL^{\min}_{2d}(f;\vb g) \coloneqq \big\{ \,\lambda \in \cLone_{2d}(\vb g) \mid \braket{\lambda}{f}=f_{\mom, d}^* \, \big\}.
    \]
\end{definition}
When $\cQ(\vb g)$ is Archimedean and $d$ is big enough, the infimum $f^*_{\mom,d}$ is attained and $\cL^{\min}_{2d}(f;\vb g)$ is nonempty, see \cite{josz_strong_2016} and {\cite[p.~101]{baldi_thesis_2022}}. 
Geometrically, $\cL^{\min}_{2d}(f;\vb g)$ is the exposed face of {the minimizers} of the MoM relaxation of order $d$.

We now introduce two convergence properties that will be central in the article.

\begin{definition}[Finite Convergence]
    We say that the SoS hierarchy $(\cQ_{2d}(\vb g))_{d\in \N}$
    (resp. the MoM hierarchy {$(\cL_{2d}(\vb g))_{d\in \N}$)} has the
    \emph{Finite Convergence} property for $f$ if {$\exists d\in \N$ such that $f^*_{\sos,d}=f^*$ (resp. $f^*_{\mom,d}=f^*$)}.
\end{definition}

Notice that if the SoS hierarchy has finite convergence then the MoM hierarchy has finite convergence too, since $f^*_{\sos,d}\le f^*_{\mom,d}\le f^*$. Moreover, if $f^*_{\mom, d} = f^*$ then $\cL^{\min}_{2d}(f;\vb g) = \{ \, \lambda \in \cLone_{2d}(\vb g) \mid \braket{\lambda}{f}=f^* \, \}$.

\begin{definition}[SoS Exactness]
    We say that the SoS hierarchy $(\cQ_{2d}(\vb g))_{d\in \N}$ is \emph{exact} for $f$ if it has the finite convergence property and {$\exists d \in \N$ such that $f-f^* \in \tqgen{g}{2d}$ (in other words $\sup=\max$ in the definition of $f^*_{\sos,d}$)}.
\end{definition}

{Notice that, by definition, if the conditions for SoS and MoM finite convergence, and SoS exactness, hold true at order $d$, then they hold true at order $k$ for all $k\ge d$.}

For the moment hierarchy we can ask the property that every truncated functional minimizer is {represented by} a measure. {This is the most natural condition which implies the finite convergence of the MoM hierarchy. We will show in \Cref{sec::flat_tru} that flat truncation (the condition used in practice to verify the finite convergence and to extract the minimizers) implies that every truncated functional minimizer is represented by a measure.}

{In the following we denote $\cMone(S)$ the finite positive Borel probability measure supported on $S$, which are identified with the induced moment linear functionals $f \mapsto \int f \dd \mu$ acting on polynomials. $\cMone(S)^{[k]}$ denotes the restrictions of such linear functionals to polynomials of degree $\le k$.}
\begin{definition}[MoM Exactness]
    We say that the MoM hierarchy $(\cL_{2d}(\vb g))_{d \in \N}$ is
    \emph{exact} for $f$ if:
    \begin{itemize}
        \item it has the finite convergence property, and
        \item for every $k \in \N$, there exists $d=d(k) \in \N$ such that every truncated functional minimizer
          is {represented by} a probability measure supported on $S$, i.e. $\cL_{2d}^{\min}(f;\vb g)^{[k]} \subset
          \cMone(S)^{[k]}$.
    \end{itemize}
\end{definition}

Notice that, {in contrast with the previous definitions,} in the definition of MoM exactness we require the property $\cL_{2d(k)}^{\min}(f;\vb g)^{[k]} \subset \cMone(S)^{[k]}$ to hold for every $k$, and in general the fact that the property is verified for particular $k$ does not imply that it holds for every $k$.

We show now an example where we investigate the properties of finite convergence and exactness.
\begin{example}
\label{ex::BHC}
Consider the problem of minimizing $f = X^2$ on the semialgebraic set $S = \cS(\vb g) = \cS(1-X^2-Y^2, X+Y-1) \subset \R^2$ defined by $g_1 = 1-X^2-Y^2$ and $g_2 = X+Y-1$. Clearly, the minimum is $f^* = 0$ and the only minimizer is $(0,1)$. Notice that $f-f^* = X^2 \in \cQ_2(1-X^2-Y^2, X+Y-1)$ 
and therefore $f^*_{\sos, 1} = f^*_{\mom, 1} = f^* = 0$, we have finite convergence and the SoS hierarchy is exact.

We now investigate MoM exactness. If a truncated moment sequence $\lambda$ is {represented by} a probability measure $\mu \in \cMone(S)$ such that $\int f \dd{\mu} = f^*$, then the support of $\mu$ should be contained in the set of minimizers $S^{\min} = \{ \, (0,1) \, \}$ of $f$. Thus $\mu = \eval_{(0,1)}$ is the evaluation at $(0, 1)$ (or in other words, the Dirac measure concentrated at $(0,1)$). Its moments are easily computed: $\mu_{00} =1$, $\mu_{10} =0$, $\mu_{01} =1$, $\mu_{20} =0,{\mu_{11}=0, \mu_{02}=1}$. 

Analyzing the constraints on the degree one and two moments of an optimal  moment sequence $\lambda \in \cL_2^{\min}(f;\vb g)$, where
\begin{align*}
    \cL_2^{\min}(f;\vb g) & = \{ \, \lambda \in \R[\vb X]_2^* \mid H^1_{\lambda} \succcurlyeq 0, \ H^0_{g_1 \star \lambda}\succcurlyeq 0, \ H^0_{g_2 \star \lambda}\succcurlyeq 0, \braket{\lambda}{1}=1, \braket{\lambda}{f} = f^*=0\} \\
    & = \{ \, \lambda \in \R[\vb X]_2^* \mid \begin{pmatrix}\lambda_{00} & \lambda_{10} & \lambda_{01} \\ \lambda_{10} & \lambda_{20} & \lambda_{11}\\
    \lambda_{01} & \lambda_{11} & \lambda_{02}\end{pmatrix} \succcurlyeq 0, \ \lambda_{00} - \lambda_{20} - \lambda_{02} \ge 0, \  \lambda_{10}+ \lambda_{01} - \lambda_{00} \ge 0, \lambda_{00} = 1,  \lambda_{20} = 0\},
\end{align*}
we deduce that  $\lambda_{00} = 1$,  $\lambda_{10} = 0$, $\lambda_{01} = 1$, $\lambda_{20} = 0$,  $\lambda_{11} = 0$ and $\lambda_{02} = 1$: this shows that the only element of $\cL_2^{\min}(f;\vb g)$ is $\lambda = \eval_{(0,1)}^{[2]}$.
In particular notice that $\braket{\lambda}{X^2} = \braket{\lambda}{(Y-1)^2}=0$. 

For any order $d\ge 1$ and any element $\lambda \in \cL_{2d}^{\min}(f;\vb g)$, its truncation $\lambda^{[2]}$ is in 
$\cL_{2}^{\min}(f;\vb g)$ since $\braket{\lambda^{[2]}}{X^2} = \braket{\lambda}{X^2}=0$. This implies that $\braket{\lambda}{X^2} = \braket{\lambda}{(Y-1)^2}=0$ and thus $\forall p \in \R[\vb X]_{d}, \braket{\lambda}{X\, p}= \braket{\lambda}{(Y-1) p}=0$, see e.g. \cite[Lem.~3.12]{lasserre_moment_2013}.
We deduce from \Cref{prop:dual eval} that the moments of $\lambda^{[d]} = \eval_{(0,1)}^{[d]}$ are {represented by} the Dirac measure $\eval_{(0,1)}$. Therefore the MoM hierarchy is exact.

Another equivalent way to certify MoM exactness is to check flat truncation (see \Cref{def:flat truncation}). For $\lambda \in \cL_{2d}^{\min}(f;\vb g)$ with $d\ge 2$, we have computed the moments of degree $\le 2$.
Since the moment matrices in degree $\le 2$:
\[
    H_{\lambda}^0 = \begin{pmatrix} 1 \end{pmatrix}, \quad H_{\lambda}^1 = \begin{pmatrix}1 & 0 & 1 \\ 0 & 0 & 0\\ 1 & 0 & 1 \end{pmatrix}
\]
have the same rank, the flat extension property is satisfied. This certifies that $\lambda^{[2]} =  \eval_{(0,1)}^{[2]}$ is {represented by} a measure supported at the minimizer of $f$ on $S$ and the MoM hierarchy is exact, see \Cref{thm::flat_implies_exact}.
\end{example}

In practice, to check the finite convergence, one tests the flat
extension or the flat truncation property (see
\cite{CurtoFlatExtensionsPositive1998},
\cite{Laurentgeneralizedflatextension2009},
\cite{nie_certifying_2013}).
But flat truncation certifies MoM exactness, and not only
finite convergence. We will investigate flat truncation for POPs in \Cref{sec::flat_tru}, where it is also described in more detail its practical importance.

Notice that in the previous example the rank condition is satisfied by the full sequence of moments of $\lambda \in \cL_2^{\min}(f;\vb g)$ (i.e. general this is not true, as the high degree moments may be increasing the rank of the moment matrix, see for in instance \cite[Ex.~1.1]{nie_certifying_2013} {and also \cite{Quijorna2021}}. Therefore it is necessary to discard the high degree moments, i.e. to consider $\cL_{2d}^{\min}(f;\vb g)^{[t]}$, for some $t \le 2d$, instead of simply $\cL_{2d}^{\min}(f;\vb g)$. This implies that we look for rank conditions on the moment matrix of the truncated moment sequence (i.e. we use the \emph{flat truncation}, see \Cref{def:flat truncation}).
{This point of view, of considering restrictions of linear functionals to a smaller degree, is a key concept in the manuscript, and it is ubiquitous in all our main results.}

We finally recall that we are assuming $S^{\min}\neq \emptyset$ (and in particular $f^*$ is finite). Notice also that if strong duality holds, then there is no duality gap and SoS finite convergence is equivalent to MoM finite convergence.
\subsection{Examples and counterexamples}

{While for generic, regular polynomial optimization problems we expect both SoS exactness \cite{nie_optimality_2014} and MoM exactness with flat truncation (see \Cref{sec::BHC} and \cite{nie_certifying_2013}), these two notions are in general independent.}

{We collected and described in details known and new examples of these possible pathological behaviours in \Cref{appendix}. In particular, we  produce new examples of optimization problems over finite semialgebraic sets, defined by Archimedean {quadratic} modules, where we do not have finite convergence and flat truncation, see \Cref{ex:no_convergence_zero} and \Cref{ex::finite_not_exact}}. We summarize these examples in \Cref{tab::examples} in terms of
the properties of finite convergence (SoS f.c. and MoM f.c.)
exactness (SoS ex. and MoM ex.), flat truncation (Flat tr.) and the dimension of the semialgebraic set $S$.
\begin{table}
\caption{Summary of convergence results.}
\label{tab::examples}
\begin{center}
\begin{tabular}{|c | c | c | c | c | c | c |}\hline
  Expl. & SoS f. c. & SoS ex. & MoM f. c. & Flat tr. & MoM ex. 
  & $\dim S$\\
  \hline
  {\ref{ex:no_convergence_zero}} & {NO} & {NO} & {NO} & {NO} & {NO} 
  & {0}\\
  \hline
  \ref{ex:curvegenus1} & NO & NO & NO & NO & NO 
  & 1\\
  \hline
  \ref{ex:dim2} & NO & NO & NO & NO & NO 
  & 2\\
  \hline
  \ref{ex:dim3} & NO & NO & NO & NO & NO 
  & $\ge 3$\\
  \hline
  \ref{ex:cylinder} & YES & YES & YES & NO & NO 
  & $\ge 3$\\
  \hline
  \ref{ex::finite_not_exact} & YES & YES & YES & NO & NO 
  & 0\\
  \hline
  \ref{ex::nonrad_grad} & YES & NO & YES & YES & YES 
  & 0\\
  \hline
  \ref{ex::fin_var} & YES & NO & YES & YES  & YES 
  & 0\\ \hline
\end{tabular}
\end{center}
\end{table}
\section{Geometry of Moment Representations}\label{sec:Geometry}
Determining whether a linear functional on the polynomial ring {is represented by} a measure is the so-called \emph{Moment Problem}.
By Haviland's theorem (see e.g. \cite[Th. 3.1.2]{marshall_positive_2008} and \cite[Th. 1.12]{schmudgen_moment_2017}) an infinite pseudo-moment sequence or a linear functional $\lambda \in \R[\vb X]^*$ {is represented by} a measure, if and only if $\lambda$ is takes {nonnegative values when applied to nonnegative polynomials}.  Since checking this is a computationally hard task, a motivation supporting Sum of Squares relaxations is to find (proper) subsets of positive polynomials that have the same property, but chosen in such a way that checking this condition is easy. An important result in this direction is the theorem of Putinar, refining the result of Schmüdgen.
\begin{theorem}[{\cite{putinar_positive_1993}}]
    \label{thm::mom_com}
    Let $\cQ(\vb g)$ be an Archimedean quadratic module and $S=\cS(\vb g)$. Let $\lambda \in \R[\vb X]^*$. If $\braket{\lambda}{g} \ge 0$ for all $q \in \cQ(\vb g)$, then $\lambda \in \cM(S)$ is {represented by} a measure.
\end{theorem}

Hereafter we analyze in detail the properties of finite dimensional cones of {\emph{truncated}} pseudo-moment sequences, and we investigate the truncated moment problem. 
\Cref{sec:truncatedmomrep} and \Cref{ref:annihilator} contains the main technical tools of the paper. 
We provide a new and explicit description of the dual of the hierarchy of truncated moment sequences, in terms of a quadratic module (\Cref{prop::qtilde}),
and consequently prove properties of the cones $\cL_d(\vb g)$ (\Cref{lem::dual_support}) and of their generic elements (\Cref{thm::generic_t}). Finally we apply these results to the zero-dimensional case (\Cref{thm::fin_set_qua}) and we investigate the connections with the flat truncation property (\Cref{subsec::interpolation_bases}).

\subsection{Truncated moment representations}\label{sec:truncatedmomrep}
 
For a finitely generated quadratic module $Q=\cQ(\vb g)\subset \R[\vb X]$, we have $\cL_k(\vb g) = \tqgen{g}{k}^{\vee} = \overline{\tqgen{g}{k}}^{\vee}$ and $\cL_k(\vb g)^\vee=\overline{\tqgen{g}{k}}$, where $\relax^{\vee}$ denotes the dual cone and the closure is taken w.r.t. the euclidean topology on $\R[\vb X]_k$. Thus the following definition is natural for the study of the MoM relaxations.
\begin{definition}
    Let $Q = \cQ(\vb g)$ be a finitely generated quadratic module. We
    define $\widetilde{Q} = \bigcup_d \overline{\tqgen{\vb g}{d}}$.
\end{definition}

Notice that $\widetilde{Q}$ depends a priori on
the generators $\vb g$ of $Q$. We will prove that $\widetilde{Q}$ is a finitely generated quadratic
module and that it does not depend on the particular choice of
generators. Moreover notice that $ Q \subset
\widetilde{Q} = \bigcup_d \overline{\tqgen{g}{d}}\subset
\overline{\bigcup_d \tqgen{g}{{d}}}=\overline{Q} \subset \R[\vb X]$, {where the closure in $\R[\vb X]$ is taken with respect to the finest locally convex topology, see e.g. \cite[Sec.~3.6]{marshall_positive_2008}}. We also remark that these inclusions can
be strict, as we will discuss in this section ({see in particular \Cref{ex::scheiderer}}).

Recall that $\supp Q = Q \cap -Q$ is an ideal if $Q$ is a quadratic module. {Recall also that, if $J \subset \R[\vb X]$, then $J_{k} = J \cap \R[\vb X]_k$ denotes the intersection with the space of polynomials of degree $\le k$}.
In this section, we also denote $\supp \tqgen{g}{k} = \tqgen{g}{k} \cap - \tqgen{g}{k} \subset \R[\vb X]_k$ and $\supp \overline{\tqgen{g}{k}} = \overline{\tqgen{g}{k}} \cap -\overline{\tqgen{g}{k}}\subset \R[\vb X]_k$.
\begin{lemma}
    \label{lem::support}
    Let $Q=\cQ(\vb g)$ and $J = \sqrt[\R]{\supp Q}$. Then for every $d \in \N$ there
    exists $k \ge d$ such that $J_d  \subset \overline{
    \tqgen{g}{k}}$.
\end{lemma}
\begin{proof}
    We denote $Q_{[d]}\coloneqq\tqgen{g}{d}$. Let $m$ be big enough such that $\forall f \in J= \sqrt[\R]{\supp Q}= \sqrt{\supp Q}$ we have: $f^{2^m} \in \supp Q$
    (if $\sqrt{\supp Q}=(h_1,\dots,h_t)$ and $h_i^{a_i} \in \supp Q$, we can take
    $m$ such that $2^m \ge a_1+\dots+a_t$).
    Let $f\in J_{d}$ with $\deg f \leq d$. Then $f^{2^m} \in \supp \cQ_{[k']}\subset Q_{[k']}$ for $k'\in \N$ big enough. Using the identity \cite[Rem.~2.2]{scheiderer_non-existence_2005}:
    \[
        m - a = (1- \frac{a}{2})^2+(1-\frac{a^2}{8})^2+(1-\frac{a^4}{128})^2+ \dots + (1 - \frac{a^{2^{m-1}}}{2^{2^m-1}})^2 - \frac{a^{2^m}}{2^{2^{m+1}-2}},
    \]
    substituting $a$ by $-\frac{mf}{\epsilon}$ and multiplying by
    $\frac{\epsilon}{m}$, we have that $\forall \epsilon >0$, $f +
    \epsilon \in  Q_{[k]}$ for $k=\max \{k', 2^m d\}$ (the degree of the representation of $f+\epsilon$ does not depend on $\epsilon$). This implies that $f \in \overline{\cQ_{[k]}}$.
  \end{proof}

We can now describe $\widetilde{Q} = \bigcup_d \overline{\tqgen{\vb g}{d}}$.
\begin{theorem}
    \label{prop::qtilde}
    Let $Q=\cQ(\vb g)$ be a finitely generated quadratic module and
    let $J = \sqrt[\R]{\supp Q}$. Then $\widetilde{Q}=Q+J$ and $\supp \widetilde{Q} = J$. In particular, $\widetilde{Q}$ is a finitely generated quadratic module and does not depend on the particular choice of generators of $Q$.
\end{theorem}
\begin{proof}
     We denote $\tqgen{g}{d} = Q_{[d]}$. By \cite[Lem. 4.1.4]{marshall_positive_2008} $Q_{[d]}+J_d$ is closed in $\R[\vb X]_d$, thus
    $\overline{Q_{[d]}}\subset Q_{[d]}+J_d$. Taking
    unions we prove that $\widetilde{Q}\subset Q+J$.

    Conversely by \Cref{lem::support} for $d\in \N$ and $k\ge d \in
    \N$ big enough,  $J_d \subset \overline{Q_{[k]}}$. Then, we have $Q_{[d]}+J_d \subset Q_{[k]}+\overline{Q_{[k]}} \subset \overline{Q_{[k]}} + \overline{Q_{[k]}}\subset \overline{Q_{[k]}}$.
    Taking unions on both sides gives $Q + J \subset \widetilde{Q}$.

    Finally $\supp \widetilde{Q} = \supp (Q + J) = J$ by \cite[Lem. 3.16]{scheiderer_non-existence_2005}.
\end{proof}

\begin{remark}
    We proved that $\widetilde{Q}=Q+\sqrt[\R]{\supp Q}$.
    We also have $\supp \widetilde{Q}=\sqrt[\R]{\supp Q}$ so that if
    $\supp Q$ is not real radical then
    $Q \subsetneq \widetilde{Q}$. \Cref{ex::fin_var} is such a case
    where $\supp Q\neq \sqrt[\R]{\supp Q}$. We notice that, by \Cref{prop::qtilde} and \cite[Th.~3.17]{scheiderer_non-existence_2005}, if $Q$ is \emph{stable},\footnote{$\cQ(\vb g)$ is stable if $\forall d \in \N$ there exists $k \in \N$ such that $\cQ(\vb g) \cap \R[\vb X]_d \subset \tqgen{g}{k}\cap \R[\vb X]_d$.} then $\widetilde{Q}=\overline{Q}$.
    But the inclusion $\widetilde{Q} = Q + \sqrt{\supp Q} \subset \overline{Q}$ can be strict, as shown by the following example.
  \end{remark}
 
\begin{example}[{\cite[Ex. 3.2]{scheiderer_distinguished_2005}, \cite[Rem. 3.15]{scheiderer_non-existence_2005}, \Cref{ex:no_convergence_zero}, \Cref{ex::finite_not_exact}}]
    \label{ex::scheiderer}
    Let $Q=\cQ(1-X^2-Y^2, -XY, X-Y, Y-X^2)\subset \R[X,Y]$. Notice that $S=\cS(Q)=\{\vb 0\}$ and that $Q$ is Archimedean. Therefore, by \Cref{thm::mom_com}, $\overline{Q}=\cP(\{\vb 0\})$. One can verify that $\supp Q = (0)$ and that $\cI(S)=\supp \overline{Q}= (X,Y)$. Thus we have $Q + \sqrt{\supp Q} = \widetilde{Q}\subsetneq \overline{Q}$.
\end{example}
\Cref{prop::qtilde} suggests the idea that, when we consider the MoM hierarchy, we are extending the quadratic module $\cQ(\vb g)$ to $\cQ(\vb g, \pm \vb h)$, where $\vb h$ are generators of $\sqrt[\R]{\supp \cQ(\vb g)}$. We specify this idea in \Cref{lem::dual_support}, \Cref{prop:radical closed} and \Cref{thm::generic_t}, investigating the relations between the truncated parts of $\cL_d(\vb g)$. {In the following, we denote $\langle \vb h \rangle_{t} = \{\, \sum_i p_i h_i \colon p_i \in \R[\vb X], \  \deg(p_i h_i) \le t \, \}$  the vector space of polynomials in the ideal $(\vb h)$ generated in degree $\le t$}.
\begin{lemma}
\label{lem::dual_support}
    Let $J=\sqrt[\R]{\supp \cQ(\vb g)}$. If $(\vb h) \subset J$, $\deg \vb h \le t$, then $\exists d \ge t$ such that $\langle \vb h \rangle_{t} \subset \overline{\tqgen{g}{d}}$. In this case:
    \[
        \cL_d(\vb g)^{[t]} \subset \cL_t(\vb g, \pm \vb h)\subset \cL_t(\vb g),
    \]
    and in particular $\cL_d(\vb g)^{[t]} \subset \cL_t(\pm \vb h)$.
    Moreover, $\cL_{d+2k}(\vb g)^{[t+k]} \subset \cL_{t+k}(\pm \vb h)$ for all $k \in \N$.
\end{lemma}
\begin{proof} {Let $\langle \vb h \rangle_{t}$ be the vector space of polynomials in the ideal $(\vb h)$ generated in degree $\le t$}.
        By \Cref{lem::support}, $\tigen{\vb h}{t} \subset (\vb h)_{t} \subset \overline{\tqgen{g}{d}}$ for some $d \ge t$. Let $h \in \vb h$ and $f \in \R[\vb X]_{t-\deg h}$. Then $\pm fh \in \overline{\tqgen{g}{d}}$, and for $\lambda \in \cL_d(\vb g)$, we have $\braket{\lambda^{[t]}}{fh}=\braket{\lambda}{fh}=0$, i.e. $\cL_d(\vb g)^{[t]} \subset \cL_t(\vb g, \pm \vb h)$. The other inclusion $\cL_t(\vb g, \pm \vb h)\subset \cL_t(\vb g)$ follows by definition.
        
        For the second part, notice that $\tigen{\vb h}{t+k} \subset \overline{\tqgen{g}{d+2k}}$. Indeed, if $p \in \tigen{\vb h}{t+k}$ then $p = \sum_i \vb X^{\alpha(i)} p_i$, where $p_i \in \tigen{\vb h}{t} \subset \overline{\tqgen{g}{d}}$ and $\abs{\alpha(i)} \le k$. Writing $\vb X^{\alpha(i)} = (\frac{\vb X ^{\alpha(i)}+1}{2})^2-(\frac{\vb X ^{\alpha(i)}-1}{2})^2$, we deduce that $p = \sum_i (\frac{\vb X ^{\alpha(i)}+1}{2})^2 p_i + (\frac{\vb X ^{\alpha(i)}-1}{2})^2 (-p_i) \in \overline{\tqgen{g}{d+2k}}$, i.e. $\tigen{\vb h}{t+k} \subset \overline{\tqgen{g}{d+2k}}$. Then we can conclude the proof as in the first part.
\end{proof}

\begin{remark}
    \Cref{lem::dual_support} says that the MoM hierarchy $(\cL_{2d}(\vb g))_{d\in \N}$ is equivalent to the MoM hierarchy $(\cL_{2d}(\vb g, \pm \vb h))_{d\in \N}$, where $(\vb h) = \sqrt[\R]{\supp \cQ(\vb g)}$.
    \Cref{lem::dual_support} is an algebraic result, in the sense that
    $\supp \cQ(\vb g)$ may be unrelated to the geometry $\cS(\vb g)$ that $\vb g$
    defines. If some additional conditions hold (namely if we have only
    equalities, or a preordering, or a small dimension), it can however
    provide geometric characterizations. Recall that we denote $\Pi \vb g$ all the products of the $g_i$'s.
\end{remark}
\begin{corollary}
    \label{cor::tmomgeo_preordering}
    Suppose that $\cS(\vb g) \subset \cV_{\R}(\vb h)$. Then for every $t_0\ge \deg \vb h$ there exists $t_1 \ge t_0$ such that:
    \[
        \cL_{t_1}(\Pi \vb g)^{[t_0]} \subset \cL_{t_0}(\pm \vb h).
    \]
    In particular this holds when $(\vb h) = \cI(\cS(\vb g))$.
    
     Moreover, $\cL_{t_1+2k}(\vb g)^{[t_0+k]} \subset \cL_{t_0+k}(\pm \vb h)$ for all $k \in \N$.
\end{corollary}
\begin{proof}
    Recall the Real Nullstellensatz, see e.g. \cite[Note~2.2.2~(vi)]{marshall_positive_2008}:
    $\cI(S(\vb g)) = \sqrt[\R]{\supp \cO(\vb g)}$. Then
    $\cS(\vb g)  \subset \cV_{\R}(\vb h)$ if and only if $\sqrt[\R]{(\vb h)} = \cI(\cV_{\R}(\vb h)) \subset \cI(\cS(\vb g))= \sqrt[\R]{\supp \cQ(\Pi \vb g)} = \sqrt[\R]{\supp \cO(\vb g)}$. We can then conclude applying \Cref{lem::dual_support}.
\end{proof}
\begin{corollary}
    \label{cor::tmomgeo_support}
    Let $Q=\cQ(\vb g)$. Suppose that $\cS(\vb g) \subset \cV_{\R}(\vb {h})$ and $\dim \frac{\R[\vb{X}]}{\supp{Q}} \le 1$.
    Then for every $t_0 \ge \deg \vb h$ there exists $t_1 \ge t_0$ such that $(\vb h)_{t_0} \subset \overline{\tqgen{g}{t_1}}$. In this case:
    \[
        \cL_{t_1}(\vb g)^{[t_0]} \subset \cL_{t_0}(\pm \vb{h}),
    \]
    and in particular this holds when $(\vb {h}) = \cI(\cS(\vb g))$.
    
     Moreover, $\cL_{t_1+2k}(\vb g)^{[t_0+k]} \subset \cL_{t_0+k}(\pm \vb h)$ for all $k \in \N$.
\end{corollary}
\begin{proof}
We prove it as \Cref{cor::tmomgeo_preordering}, using \cite[cor.~7.4.2~(3)]{marshall_positive_2008}:
\begin{equation}\label{thm::marshall}
    \dim \frac{\R[\vb X]}{\supp \cQ(\vb g)}\le 1 \Rightarrow \cI(\cS(\vb g)) = \sqrt[\R]{\supp \cQ(\vb g)}
\end{equation}
instead of the Real Nullstellensatz.
\end{proof}

We conclude this section discussing how the closure properties of truncated quadratic modules are related with strong duality for Lasserre's hierarchies. We briefly recall existing results in this direction.

\begin{theorem}[Strong duality]
\label{thm::strong_duality}
Let $Q=\cQ(\vb g)$ be a quadratic module and $f$ the objective function. Then:
\begin{enumerate}
    \item if $\supp Q = (0)$ then $\forall d$: $f^*_{\sos,d}$ is attained (i.e. $f-f^*_{\sos,d} \in \tqgen{g}{2d}$) and $f^*_{\sos,d}=f^*_{\mom,d}$ \cite[Prop.~10.5.1]{marshall_positive_2008};
    \item if for some $r\in \R$, $r^2 - \norm{\vb X}^2 \in \vb g$ then $f^*_{\sos,d}=f^*_{\mom,d}$ for all $d$ \cite{josz_strong_2016}.
\end{enumerate}
\end{theorem}

\begin{remark}
    \cite{josz_strong_2016} applies when the ball constraint $r^2 - \norm{\vb X}^2$ appears {explicitly} in the description of $S$. But if we consider a problem with MoM finite convergence and such that $\cQ(\vb g)$ is Archimedean, then we can use \cite{josz_strong_2016} to prove that we have also SoS finite convergence. Indeed, if $\cQ(\vb g)$ is Archimedean there exists $r, t$ such that $r^2 - \norm{\vb X}^2 \in \tqgen{g}{2t}$. This means that $\cQ_{2d}(\vb g, r^2 - \norm{\vb X}^2) \subset \tqgen{g}{2d+2t}$. If we define:
    \begin{itemize}
        \item $\displaystyle f^*_{\sos,d}  = \sup \big\{ \, \lambda \in \R \mid f-\lambda  \in \tqgen{g}{2d} \,\big\}$
        \item $\displaystyle f^{*'}_{\sos,d}  = \sup \big\{ \, \lambda \in \R \mid f-\lambda  \in \cQ_{2d}(\vb g , r^2 - \norm{\vb X}^2) \,\big\}$
    \end{itemize}
    and $f^*_{\mom,d}$, $f^{*'}_{\mom,d}$ the corresponding MoM relaxations, then:
    \[
        f^*_{\mom,d} \le f^{*'}_{\mom,d} = f^{*'}_{\sos,d} \le f^*_{\sos,d+t} \le f^*.
    \]
    Then finite convergence of the MoM hierarchy implies finite convergence of the SoS one. 
\end{remark}

We now prove a strong duality result, that will be useful to analyze finite convergence and exactness in \Cref{appendix}. This result, very similar to a result in \cite{marshall_optimization_2003}, generalizes the condition $\supp Q = (0)$ in \Cref{thm::strong_duality}. 

{For this generalization, we need the definition of \emph{graded basis}. We say that $\vb h$ is a graded basis of $I = (\vb h)$ if $\langle \vb h \rangle_t = I\cap \R[\vb X]_t=I_t$ for all $t\in \N$, where $\langle \vb h \rangle_t$ denotes the vector space of polynomials in $I = (\vb h)$ generated in degree $\le t$. For instance, any Groebner basis with respect to an ordering compatible with the total degree is a graded basis.}

\begin{prop} \label{prop:radical closed} 
 Let $Q=\cQ(\vb g)$ be a finitely generated quadratic module, and let $\vb h$ be a {graded basis} of $\sqrt[\R]{\supp Q}$. Then for any $d$ we have {that} $\cQ_d(\vb g , \pm \vb h) = \overline{\cQ_d(\vb g , \pm \vb h)}$ is closed.
  Moreover, {if $\cS(\vb g) \neq \emptyset$} and we consider the {relaxations $\cQ_{2d}(\vb g , \pm \vb h)$ and $\cL_{2d}(\vb g, \pm \vb h)$ (extensions of $\cQ_{2d}(\vb g)$ and $\cL_{2d}(\vb g$) using the generators of $\sqrt[\R]{\supp Q}$)},  then for any $f \in \R[\vb
  X]$, there is no duality gap: $f^{*}_{\sos,d}=f^{*}_{\mom,d}$. In this case, if $f^{*}_{\sos,d} > -\infty$, then $f^{*}_{\sos,d}$ is attained (i.e. $f-f^{*}_{\sos,d} \in \cQ_{2d}(\vb g , \pm \vb h)$).
\end{prop}
\begin{proof}
{Consider the quotient map $\phi \colon \R[\vb X]_{2d} \to \R[\vb X]_{2d} \big/ I_{2d}$. By \cite[Lem. 4.1.4]{marshall_positive_2008}, the image of $\cQ_{2d}(\vb g , \pm \vb h)$ under $\phi$ is
closed. Since $\vb h$ is a graded basis,
\[
    \phi^{-1}(\phi (\cQ_{2d}(\vb g , \pm \vb h))) = \cQ_{2d}(\vb g , \pm \vb h) + I_{2d} = \cQ_{2d}(\vb g) + \langle \vb h \rangle_{2d} +I_{2d} = \cQ_{2d}(\vb g) + \langle \vb h \rangle_{2d} = \cQ_{2d}(\vb g , \pm \vb h)
\]
and thus $\cQ_{2d}(\vb g , \pm \vb h)$ is closed as well.}
Therefore we have
$\cL_{2d}(\vb g, \pm \vb h)^{\vee}= (\cQ_{2d}(\vb g , \pm \vb h))^{\vee \vee}= \overline{\cQ_{2d}(\vb g , \pm \vb h)} = \cQ_{2d}(\vb g , \pm \vb h)$,
from which we deduce that there is {no} duality gap, by classical
convexity arguments, as follows.

{ If $f \in \R[\vb X]$ and $\cS(\vb g) \neq \emptyset$, let $H=\big\{ \, a \in \R \mid f-a  \in \cQ_{2d}(\vb g , \pm \vb h) \,\big\}$. By construction, $H$ is empty or an interval unbounded from below. 
We prove now that it is always bounded from above.
Since $\cS(\vb g) \neq \emptyset$ then $f^* < +\infty$, and for all $a > f^*$ there exists $x \in \cS(\vb g)$ such that $f(x) < a$. Then $f(x) - a < f(x) - f^*$, and thus $f - a \notin \cQ(\vb g, \pm \vb h)$ and $a\not\in H$. Therefore $H$ is bounded from above.}

{If $H = \emptyset$, then $f^{*}_{\sos,d} = \sup H = -\infty$. If $f^*_{\mom, d} > -\infty$, then $\exists M \in \R$ such that $\braket{\lambda}{f} \ge -M > -\infty$ for all $\lambda \in \cLone_{2d}(\vb g, \pm \vb h)$. Using \cite[Lem.~1.3.9]{baldi_thesis_2022}, we deduce that $\braket{\lambda}{f+M} \ge 0$ for all $\lambda \in \cL_{2d}(\vb g, \pm \vb h)$, and thus $f+M \in \cL_{2d}(\vb g , \pm \vb h)^\vee = \cQ_{2d}(\vb g , \pm \vb h)$, a contradiction to $H = \emptyset$. Therefore, $f^*_{\mom, d} = -\infty = f^*_{\sos, d}$.}

{If $H \neq \emptyset$, since $\cQ_{2d}(\vb g , \pm \vb h)$ is closed and $H$ bounded from above, $f^{*}_{\sos,d} = \sup H = \sup \big\{ \, a \in \R \mid f-a  \in \cQ_{2d}(\vb g , \pm \vb h) \,\big\}$ is attained.
If $f^{*}_{\sos,d} < f^{*}_{\mom,d}$, then $f-f^{*}_{\mom,d} \not\in
\cQ_{2d}(\vb g , \pm \vb h)$. Thus there exists a separating functional $\lambda \in
\cL_{2d}^{(1)}(\vb g, \pm \vb h)$ such
that $\braket{\lambda}{f-f^{*}_{\mom,d}}<0$, which implies that
$\braket{\lambda}{f}< f^{*}_{\mom,d}$ in contradiction with the
definition of $f^{*}_{\mom,d}$. Consequently, $f^{*}_{\sos,d}=f^{*}_{\mom,d}$.}

{Finally, if $f^*_{\sos, d} > - \infty$ then $H \neq \emptyset$, and we can conclude as before that $f^*_{\sos, d}$ is attained.}
\end{proof}

\subsection{Annihilators of truncated moment sequences}\label{ref:annihilator}

Recall that the annihilator $\ann_t(\lambda)$ is the kernel of the moment matrix of $\lambda$ in degree $\leq t$ (or of the Hankel operator).
With the characterization of $\widetilde{Q}$ we can now describe these kernels of moment matrices associated to truncated positive linear functionals.

We recall the definition of genericity in the truncated setting and equivalent characterizations.

\begin{definition}
\label{def::generic_t}
    Let $C \subset \R[\vb X]_{d}^*$ be a convex set.
    We say that $\lambda^* \in C$ is \emph{generic} in $C$ if
    $\rank H_{\lambda^*}^{\floor {\frac d 2}}=\max \{ \rank H_{\eta}^{\floor {\frac d 2}} \mid \eta \in C \}$.
\end{definition}
In particular, we will consider generic $\lambda^*$ in the following convex sets, {which are used in polynomial optimization}:
\begin{itemize}
    \item $C = \cL_{2d}(\vb g)$, the cone of positive linear functionals or feasible pseudo-moment sequences of Lasserre's moment relaxation of order $d$;
    \item $C = \cLone_{2d}(\vb g)$, the convex set defined as the section of $\cL_{2d}(\vb g)$ given by $\braket{\lambda}{1}=1$;
    \item $C = \cL_{2d}^{\min}(f;\vb g)$, the exposed face of $\cLone_{2d}(\vb g)$ defined by $\braket{\lambda}{f}=f^*_{\mom,d}$;
    \item $C = \cL_{2d}(\vb g)^{[2k]}, \cLone_{2d}(\vb g)^{[2k]}, \cL_{2d}^{\min}(f;\vb g)^{[2k]}$,
    the restrictions of the positive linear functionals of the above sets to $\R[\vb X]_{2k}$.
\end{itemize}
We will use generic linear functionals to recover the minimizers when we have the flat truncation property.
They can be characterized as follows, see \cite[Prop. 4.7]{lasserre_moment_2013} and \cite[Prop.~3.4.10]{baldi_thesis_2022}.
\begin{prop}
    \label{prop::genericity}
    Let $C$ be a convex subset of $\cL_{d}(\vb g)$ and let $\lambda \in C$. The following are equivalent:
    \begin{enumerate}
        \item $\lambda$ is generic in $C$;
        \item $\ann_{\floor {\frac d 2}} (\lambda) \subset \ann_{\floor {\frac d 2}} (\eta) \ \forall \eta \in C$;
        \item $\forall k\le d$, we have: $\rank H_{\lambda}^{\floor {\frac k 2}}=\max \{ \, \rank H_{\eta}^{\floor {\frac k 2}} \mid \eta \in C \, \}$, {i.e. $\lambda^{[k]}$ is generic in $C^{[k]}$.}
    \end{enumerate}
\end{prop}

\begin{remark}
    By \Cref{prop::genericity} notice that $\forall d' \le d$, if $\lambda^*\in \cL_{d}(\vb g)$ is generic then $(\lambda^*)^{[d']}$ is generic in $\cL_{d}(\vb g)^{[d']}$. In particular, $\ann_{\floor {\frac{d'}{2}}} (\lambda^*) \subset \ann_{\floor {\frac{d'}{2}}} (\eta) \ \forall \eta \in \cL_{d}(\vb g)$.
\end{remark}
We can show that any linear functional in the relative interior is generic.
\begin{lemma}\label{lem:generic_interior}
  Let $C$ be a convex subset of $\cL_{d}(\vb g)$. Then any linear functional $\lambda$ in the relative interior of $C \subset \cL_{d}(\vb g)$ is generic in $C$.
\end{lemma}
\begin{proof}
Let $\lambda$ be in the relative interior of $C$. For any $\lambda_1\in C$, there exists $\lambda_2\in C$  and $a_1, a_2 \in \R_{>0}$ with $a_1+ a_2 = 1$,
such that $\lambda = a_1 \lambda_1 + a_2 \lambda_2$.

Denote $k={\floor {\frac d 2}}$.
The inclusion \[ \ann_{k}(\lambda) = \ker H_{\lambda}^k = \ker H_{a_1 \lambda_1 + a_2 \lambda_2}^k \supset \ker H_{\lambda_1}^k \cap \ker H_{\lambda_2}^k = \ann_k(\lambda_1)\cap \ann_k(\lambda_2) \] is direct.
Conversely, if $f \in \ann_k(\lambda)$ then $f \star \lambda = a_1(f\star \lambda_1) + a_2 (f \star \lambda_2) =0$.
In particular, \[ 0 =\braket{f \star \lambda}{f} = a_1\braket{f \star \lambda}{f} + a_2\braket{f \star \lambda}{f} = a_1\braket{\lambda_1}{f^2} + a_2\braket{\lambda_2}{f^2} \]
Therefore $\braket{\lambda_1}{f^2} = \braket{\lambda_2}{f^2}=0$, and from \cite[Lem.~3.12]{lasserre_moment_2013}
we have $f \in  \ann_d(\lambda_1)\cap \ann_d(\lambda_2)$, proving the reverse inclusion.
We deduce that $\forall \lambda_1\in C$ we have $\ann_k(\lambda)\subset \ann_k(\lambda_1)$ , thus $\lambda$ is generic from \Cref{prop::genericity}. 
\end{proof}
As any point in the relative interior {of $C$} is generic {in C} by \Cref{lem:generic_interior}, in practice generic linear functionals can be recovered from SDP solvers based on interior point methods. Indeed, these solvers approximately follow the central path to get an approximately optimal solution (possibly using a self-dual-embedding or facial reduction) of the original SDP problem, see e.g. \cite{Halick2002} and references therein. Therefore, using an SDP solver based on interior point methods we can approximately get a point in the relative interior (see also \cite[Sec.~4.4.1]{lasserre_semidefinite_2008}).
Notice also that there are examples of cones $C \subset \cL_d(\vb g)$ where extreme points are generic in $C$, see \Cref{ex::flat_out}.

We are now ready to describe the annihilator of these generic elements.

\begin{theorem}
\label{thm::generic_t}
    If $d,t \in \N$ are big enough and $\lambda^* \in \cL_d(\vb g)$ is generic, we have $\sqrt[\R]{\supp \cQ(\vb g)}=(\ann_t (\lambda^*))$. Moreover if {$\cQ(\vb g)=\mathcal{T}(\vb g)$} is a preordering, then $(\ann_t(\lambda^*)) = {\cI}(\cS(\vb g))$.
\end{theorem}
\begin{proof}
We denote $J = \sqrt[\R]{\supp Q}$.
    Let $t\in \N$ {be} such that $J$ is generated in degree $\le t$, by the
    {graded basis}
    $\vb h= h_{1}, \ldots, h_{s}$(see above \Cref{prop:radical closed} for the definition). From \Cref{lem::support} we deduce that there exists $d \in \N$ such that $J_{2t} \subset \overline{\tqgen{\vb g}{d}}$. Let $\lambda^* \in \cL_d(\vb g)$ {be} generic.

    We first prove that $J \subset (\ann_t (\lambda^*))$. By \Cref{prop::genericity} we have $\ann_t (\lambda^*) =\bigcap_{\lambda \in \cL_d(\vb g)} \ann_t (\lambda)$. Then it is enough to prove that $J_{t} \subset \ann_t (\lambda)$ for all $\lambda \in \cL_d(\vb g)$.

    By \Cref{lem::dual_support} $\cL_d(\vb g)^{[2t]}\subset
    \cL_{2t}(\pm \vb h)\subset \tigen{\vb h}{2t}^{\perp}$. Then $\forall f \in J_{t}=\tigen{\vb h}{t},\
    \forall p \in \R[\vb X]_t, \ \forall \lambda \in \cL_d(\vb g)$,  we
    have $f\, p \in \tigen{\vb h}{2t}$ and  $\braket{\lambda^{[2t]}}{fp}=0$. This shows that $H_{\lambda}^t(f)(p)=\braket{(f \star \lambda)^{[t]}}{p}=\braket{\lambda}{fp}=0$, i.e. $f \in \ann_t (\lambda) = \ker H_{\lambda}^t$.

    Conversely, we show that $(\ann_t (\lambda^*))\subset J$ for
    $\lambda^*$ generic in $\cL_d(\vb g)$. Since $J=\supp \widetilde{Q}
    = \supp \bigcup_j \overline{\tqgen{\vb g}{j}}$ (by \Cref{prop::qtilde}) it is enough to prove that $\ann_t (\lambda^*) \subset \overline{\tqgen{g}{d}} \cap - \overline{\tqgen{g}{d}} = \supp \overline{\tqgen{\vb g}{d}}= \supp \cL_d(\vb g)^{\vee}$.

    Let $f \in \ann_{t} (\lambda^*) = \bigcap_{\lambda \in \cL_d(\vb g)} \ann_t (\lambda)$ (we use again \Cref{prop::genericity}) and let $\lambda \in  \cL_d(\vb g)$. Then $\braket{\lambda}{f}=\braket{(f\star \lambda)^{[t]}}{1}=H_{\lambda}^t(f)(1)=0$. In particular $f \in \cL_d(\vb g)^{\vee}$. We prove that $-f \in \cL_d(\vb g)^{\vee}$ in the same way. Then $f \in \cL_d(\vb g)^{\vee} \cap - \cL_d(\vb g)^{\vee} = \overline{\tqgen{g}{d}} \cap - \overline{\tqgen{g}{d}} = \supp \overline{\tqgen{\vb g}{d}}$, and finally we deduce from \Cref{def::generic_t} and \Cref{prop::qtilde} that $\ann_{t} (\lambda^*) \subset \supp \widetilde{Q}=J$.
    
    The second part follows from the first one and the Real Nullstelensatz.
\end{proof}

{\Cref{thm::generic_t} generalizes the results of \cite{lasserre_semidefinite_2008} for equations defining real algebraic varieties. In \cite{lasserre_semidefinite_2008}, in the zero-dimensional case the flat truncation criterion is used to detect equality between the ideal generated by the annihilator (or by the kernel of the moment matrix) and the real radical of the equations. As explained in \Cref{prop:flat truncation}, we cannot use flat truncation to detect equality in the positive dimensional case. Other techniques to verify that $t,d$ are big enough to generate the real radical of the support have been further investigated in \cite[Ch.~4]{baldi_thesis_2022} and \cite{baldiComputingRealRadicals2021a}.}

\Cref{thm::generic_t} shows the possibilities and the limits of MoM hierarchies.
For instance we cannot expect exactness of the MoM relaxation $\cL_{2d}(\vb g)$ for any objective function $f$ (i.e. $\cL_{2d}(\vb g)^{[k]} \subset \cM(S)^{[k]}$) if $\sqrt[\R]{\supp Q} \neq \cI(S)$: see \Cref{ex::finite_not_exact}.

\subsection{Interpolation degree and flat truncation}
\label{subsec::interpolation_bases}

In this section, we analyze the properties of moment sequences in $\cL_d(\vb g)$ when $S=\cS(\vb g)$ is finite. We will use the results in this section to study the case of finitely many minimizers in Polynomial Optimization  Problems, and in particular the flat truncation property.

Let $\Xi=\{\xi_{1}, \ldots, \xi_{r}\}\in \C^{n}$ be a finite set of (complex) points 
and
let $\cI(\Xi)=\{ \, p\in \C[X]\mid p(\xi_{i})=0 \ \forall i \in {1,\ldots,r}\,\}$ be the complex vanishing ideal of the points $\Xi$. 
It is well known {(see e.g. \cite{eisenbud_geometry_2005})} that $\Xi$ admits a family of
interpolator polynomials $(u_{i})\subset \C[\vb X]$ 
such that $u_{i}(\xi_{j})=\delta_{i,j}$, which form a basis 
$\C[\vb{X}]/\cI(\Xi)$.
The minimal degree of a family of
interpolator polynomials is called the {\em interpolation degree} of $\Xi$ and denoted $\intdeg(\Xi)$.

A classical result states that $\intdeg(\Xi)+1$ is 
the {\em Castelnuovo-Mumford regularity} of the ideal $I(\Xi)$
(see \cite[Th. 4.1]{eisenbud_geometry_2005}).
This interpolation degree $\intdeg(\Xi)$ is the minimal degree of a basis of $\C[\vb{X}]/\cI(\Xi)$. It is also the minimal degree of a monomial basis $B$ of $\R[\vb{X}]/\cI(\Xi)$. Such a minimal degree basis $B$ can be chosen so that it is {a monomial basis} stable by monomial division\footnote{{i.e. if $m\, m'\in B$ then $m$ and $m' \in B$}}.
Moreover, the ideal $\cI(\Xi)$ has a graded (resp. Grobner, resp. border) basis of degree $\intdeg(\Xi)+1$ (see e.g. \cite{bayer_criterion_1987}).

The next result shows that (positive) moment sequences orthogonal to the vanishing ideal of the points, truncated above twice the interpolation degree are {represented by} (positive) measures.
Recall that if $I$ is an ideal we denote $I_t = I \cap \R[\vb X]_t$ and $\langle \, \rangle$ denotes the linear span of a set.
\begin{prop} \label{prop:dual eval}
  Let $\Xi=\{\xi_{1},\ldots,\xi_{r}\}\subset \R^{n}$, $I=\cI(\Xi)$ its real vanishing ideal and let $\intdeg=\intdeg(\Xi)$ the interpolation degree of $\Xi$. Let $t\ge\intdeg$ and $\lambda \in \R[\vb X]^*_{t}$. 
  Then $\lambda\in I_{t}^{\perp}$ if and only if $\lambda\in \tigen{\eval_{\xi_{1}}^{[t]}, \ldots, \eval_{\xi_{r}}^{[t]}}{}$. {Moreover if $t\ge 2\intdeg$ and $\lambda\in \cL_{t}(I_{t})$, then
$\lambda\in \cone(\eval_{\xi_{1}}^{[t]}, \ldots, \eval_{\xi_{r}}^{[t]})$ and $\rank H_{\lambda}^{\floor{\frac{t}{2}}} \le r$}.
\end{prop}
\begin{proof}
  Let $u_{1}, \ldots, u_{r}\in \R[\vb X]_{t}$ be interpolation polynomials of degree $\intdeg \le t$. Consider the sequence of vector space maps:
  \begin{eqnarray*}
    0 \rightarrow I_{t} \rightarrow \R[\vb X]_{t}& \stackrel{\psi}{\longrightarrow} & \tigen{u_{1}, \ldots, u_{r}}{} \rightarrow 0\\
    p & \mapsto &  \sum_{i=1}^{r} p(\xi_{i}) u_{i},
  \end{eqnarray*}
which is exact since $\ker \psi= \{p \in \R[\vb X]_{t}\mid p(\xi_{i})=0 \}=I_{t}$. Therefore we have
  $\R[\vb X]_{t} = \tigen{u_{{1}}, \ldots, u_{{r}}}{} \oplus I_{t}$.

  Let $\lambda\in I_{t}^{\perp}$. Then $\tilde{\lambda}=\lambda -\sum_{i=1}^{r}\braket{\lambda}{u_{i}}   \eval_{\xi_{i}}^{[t]}\in I_{t}^{\perp}$ is such that $\braket{\tilde{\lambda}}{u_{i}}=0$ for $i=1, \ldots, r$. Thus, $\tilde{\lambda}\in\tigen{u_{{1}}, \ldots, u_{{r}}}{}^{\perp}\cap I_{t}^{\perp} =(\tigen{u_{{1}}, \ldots, u_{{r}}}{} \oplus I_{t})^{\perp} = \R[\vb X]_{t}^{\perp}$, i.e. $\tilde{\lambda}=0$ showing that $I_{t}^{\perp}\subset \tigen{\eval_{\xi_{1}}^{[t]}, \ldots, \eval_{\xi_{r}}^{[t]}}{}$. The reverse inclusion is direct since $I_{t}$ is the space of polynomials of degree $\le t$ vanishing at $\xi_{i}$ for $i=1, \ldots, r$.

{Assume now that $t \ge 2\intdeg$ and $\lambda\in \cL_{t}(I_{t})$.
Then $\lambda\in I_{t}^{\perp}$ and
$\braket{{\lambda}}{p^{2}}\ge 0$ for any $p^2\in \R[\vb X]_t$. By the previous analysis,
\[
  \lambda = \sum_{i=1}^{r} \omega_{i} \eval_{\xi_{i}}^{[t]}
\]
As $0\le\braket{{\lambda}}{u_{i}^{2}}= \omega_{i}$ for $i=1, \ldots, r$, we deduce that $\lambda\in \cone(\eval_{\xi_{1}}^{[t]}, \ldots, \eval_{\xi_{r}}^{[t]})$.}

{Let $s = {\floor{\frac{t}{2}}}$. We verify that the image of $H_{\lambda}^{s}: p \in \R[\vb X]_{s} \mapsto \sum_{i=1}^r \omega_{i} p(\xi_i)\,  \eval_{\xi_{i}}^{[s]}$ is {included in} $\tigen{\eval_{\xi_{1}}^{[s]}, \ldots, \eval_{\xi_{r}}^{[s]}}{}$, computing  $H_\lambda^s(u_i)$ for $i=1, \ldots, r$. Thus $\rank H_\lambda^s \le \dim \tigen{\eval_{\xi_{1}}^{[s]}, \ldots, \eval_{\xi_{r}}^{[s]}}{} = r$ since $(\eval_{\xi_{i}}^{[s]})_{i=1,\dots,r}$ is the dual basis of $(u_i)_{i=1,\dots,r}$.}
\end{proof}

We describe now a property, known as \emph{flat truncation}, which allows to test effectively if truncated moment sequences are {represented by} sums of point evaluations.
\begin{definition}[{Flat truncation, \cite[Th.~6.18]{laurent_sums_2009}, \cite[Sec.~1.2]{nie_certifying_2013}}]\label{def:flat truncation} Let $d_{\vb g}=\lceil \frac 1 2 \max_{i=1,\ldots,s} \deg (g_i) \rceil$. We say that the \emph{flat truncation} property holds for $\lambda \in \cL_d(\vb g)$ at degree $t$ if $ t\le \frac d 2 - d_{\vb g}$ and
\begin{equation}\label{eq:rank flat}
    \rank H^{t}_\lambda = \rank H^{t+d_{\vb g}}_{\lambda}.
\end{equation}
\end{definition}
{This definition coincides with the definition of flat truncation given in \cite{nie_certifying_2013}, and previously exploited in \cite{laurent_sums_2009}. The flat truncation is a modification of the flat extension criterion by Curto and Fialkov (see e.g. \cite{c73aa9e9-d41f-3dde-b8fc-1b7ea47fd432}), and is very useful in polynomial optimization to certify the finite convergence and to extract the minimizers.}

{We now investigate in more detail this rank condition for the moment matrix of $\lambda \in \cL_d(\vb g)$. In the following lemma, we reprove and extend known properties associated with flat truncation (see e.g. \cite[Th.~5.33]{laurent_sums_2009} and references therein). Our improvements are the characterization of the annihilator (or, of the kernel of the moment matrix) as the full truncated vanishing ideal of the points, and a truncation degree for which all the linear functionals coincide with sums of point evaluations.}

\begin{lemma}\label{lem:flat diracs}
If $\lambda \in \cL_d(\vb g)$ is such that $\rank H^{t}_{\lambda} = \rank H^{t+s}_{\lambda} = r$ with $t+1 \le t +s \le \frac{d}{2}$, then 
$$
\lambda^{[t + s + \frac d 2 ]} = \omega_1 \eval_{\xi_1}^{[t + s + \frac d 2 ]}+ \dots + \omega_r \eval_{\xi_r}^{[t + s + \frac d 2 ]}
$$
for some points $\xi_i\in \R^n$ and weights $\omega_i>0$, $i=1,\dots,r$. Denoting $\Xi=\{\xi_1,\dots ,\xi_r \}$, we also have $\ann_{t+s} (\lambda)=\cI(\Xi)_{t+s}$ and $\cV(\ann_{t+s} (\lambda))=\Xi$ (or, in other words, $(\ann_{t+s} (\lambda)) = \cI(\Xi)$).

Moreover, if $t\le  \frac d 2 + s -\deg({\vb g})$, where $\deg (\vb g) = \max_{i=1,\dots, s} \deg(g_i)$, the inclusion $\Xi \subset \cS(\vb g)$ holds true.
\end{lemma}
\begin{proof}
    From \cite[Th.~5.29]{laurent_sums_2009}, there exists unique $\Xi=\{\xi_1,\dots ,\xi_r \}\subset \R^n$ and $\omega_1,\dots , \omega_r > 0$ such that $\lambda^{[2(t+s)]} = \omega_1 \eval_{\xi_1}^{[2(t+s)]}+ \dots + \omega_r \eval_{\xi_r}^{[2(t+s)]}$,  $(\ann_{t+s} (\lambda))=\cI(\Xi)$ and $\cV(\ann_{t+s} (\lambda))=\Xi$. In particular $(\ann_{t+s} (\lambda))$ is a zero-dimensional ideal 
    and $\ann_{t+s}(\lambda)\subset I(\Xi)_{t+s}$. Conversely, for any $h\in I(\Xi)_{t+s}$, we have
    \[
         \braket{\lambda}{h^2} =  \braket{\lambda^{[2(t+s)]}}{h^2} = \sum_{i=1}^{r} \omega_i 
         \braket{ \eval_{\xi_i}^{[2(t+s)]}}{h^2} = \sum_{i=1}^{r} \omega_i h^2(\xi_i) = 0.
    \]
    Thus $h \in \ann_{t+s}(\lambda)$ (see \cite[Lem.~3.12]{lasserre_moment_2013}) and $I(\Xi)_{t+s}= \ann_{t+s}(\lambda)$. 
    
    As $\rank H^{t}_{\lambda} = \rank H^{t+1}_{\lambda} = r$, we deduce from above, that $(\ann_{t+1}(\lambda))=I(\Xi)$ is generated in degree $\le t+1$ and that $\theta(\Xi) \le t$. Therefore $\Xi$ has interpolator polynomials $u_1, \dots, u_r$
    of degree $\le t$.
    
    Let us show that the description of $\lambda$  on polynomials of degree $\le 2(t+s)$, can be extended to higher degree. 
    For any $h\in \ann_{t+s}(\lambda)= I(\Xi)_{s+t}$, i.e. such that  $\braket{\lambda}{h^2}=0$, and any $p \in \R[\vb X]_{\frac{d}{2}}$ we have $\braket{\lambda}{h p} = 0$. This shows that $\lambda \in (\cI(\Xi)_{t+s+\frac{d}{2}})^{\perp}$. We deduce from \Cref{prop:dual eval} that $\lambda^{[t + s+ \frac d 2 ]} \in \cone(\eval_{\xi_1},\ldots,\eval_{\xi_r})^{[t + s + \frac d 2 ]}$. This implies that $\lambda^{[t + s + \frac d 2 ]} = \omega_1 \eval_{\xi_1}^{[t + s + \frac d 2 ]}+ \dots + \omega_r \eval_{\xi_r}^{[t + s + \frac d 2 ]}$, evaluating $\braket{\lambda}{u_i}= \braket{\lambda^{[t+ s+\frac d 2]}}{u_i}=\omega_i$ at the interpolator polynomials $u_1, \dots u_r$ of $\Xi$ of degree $\le t$.
    
    We show now that $\Xi= \{ \xi_1,\dots ,\xi_r \} \subset {\cS(\vb g)}$ if $t\le  \frac d 2 + s -\deg({\vb g})$. For $i=1, \dots, r$ and $j=1, \dots, m$ the polynomial $u_i^2 g_j$ has degree $\le 2 t + \deg({\vb g}) \le t + s + \frac d 2$. Then we obtain:
    \[
        0 \le \braket{\lambda}{u_i^2 g_j} = \braket{\lambda^{[t + s + \frac d 2 ]}}{u_i^2 g_j} = \braket{\omega_1 \eval_{\xi_1}^{[t + s + \frac d 2 ]}+ \dots + \omega_r \eval_{\xi_r}^{[t + s + \frac d 2 ]}}{u_i^2 g_j} = {\omega_i}g_j(\xi_i),
    \]
    showing that $g_j(\xi_i) \ge 0$ for all $i$ and $j$, i.e. $\Xi \subset \cS(\vb g)$.
\end{proof}

\begin{remark}
    \Cref{lem:flat diracs} can be used to test flat truncation in a simpler way when $d$ is big, as we explain in the following. Assume for simplicity that $2 d_{\vb g} = \deg (\vb g)$. Then, if  $\rank H^{t}_{\lambda} = \rank H^{t+s}_{\lambda}$ with $t\le  \frac d 2 + s -\deg({\vb g})$, we have $2(t+ d_{\vb g}) = 2t + \deg(\vb g) \le t + s + \frac{d}{2} $. Then from \Cref{lem:flat diracs} we deduce that $\lambda$ restricted to polynomials of degree $\le 2(t+ d_{\vb g})$ is equal to a sum of evaluations at points of $S$ with positive weights, and the flat truncation is satisfied: $\rank H_{\lambda}^t = \rank H_{\lambda}^{t+d_{\vb g}}$. In particular, when $s=1$ and $d \ge 2t-2+2\deg(\vb g)$, $\rank H^{t}_{\lambda} = \rank H^{t+1}_{\lambda}$ implies $\rank H_{\lambda}^t = \rank H_{\lambda}^{t+d_{\vb g}}$.
\end{remark}

We now show that we can use flat truncation for generic linear functionals to describe semialgebraic sets with a finite number of points. {Results similar to \Cref{prop:flat truncation} and \Cref{thm::fin_set_qua} have been already studied in \cite{lasserre_semidefinite_2008, lasserre_moment_2013}: we discuss in detail the differences with previous results after  \Cref{thm::fin_set_qua}.}

\begin{theorem}\label{prop:flat truncation}
If a positive linear functional $\lambda^* \in \cL_{d}(\vb g)$ is such that $(\lambda^*)^{[2(t+d_{\vb g})]}$ is generic in $\cL_d(\vb g)^{[2(t+d_{\vb g})]}$ (that is  $\ann_{t+d_{\vb g}}(\lambda^*) \subset \ann_{t+d_{\vb g}}(\lambda)$ for all $\lambda \in \cL_{d}(\vb g)^{{[2(t+d_{\vb g})]}}$) and $\lambda^*$ satisfies the flat truncation property at degree $t \le \frac{d}{2} - d_{\vb g}$, then:
\begin{enumerate}
    \item $S=\cS(\vb g)=\{\xi_1,\dots ,\xi_r\}$ is non-empty and finite;
    \item $\displaystyle \cL_{d}(\vb g)^{[t + d_{\vb g} + \frac d 2 ]} = \cone (\eval_{\xi_1},\dots,\eval_{\xi_r})^{[t + d_{\vb g} +\frac d 2 ]}$;
    \item $t \ge \theta(\xi_1,\dots ,\xi_r)$ and $\ann_{t+1} (\lambda^*)=\cI(\xi_1,\dots ,\xi_r)_{t+1}= \cI(S)_{t+1}$ is the vanishing ideal of $S$ truncated in degree $t+1$.
    \item $\cI(S)_{2(t+d_{\vb g})} \subset \overline{\tqgen{g}{d}}$ and $(\ann_{t+1} (\lambda^*))= \sqrt[\R]{\supp \cQ(\vb g)} =\cI(S)$ (in particular, this ideal is zero-dimensional).
\end{enumerate}
\end{theorem}
\begin{proof}
    Let $\lambda^* \in \cL_d(\vb g)$ be such that $(\lambda^*)^{[2(t+d_{\vb g})]}$ is generic in $\cL_d(\vb g)^{[2(t+d_{\vb g})]}$, and assume that $\rank H^{t}_{\lambda^*} = \rank H^{t+d_{\vb g}}_{\lambda^*}$ with $t\le  \frac d 2 - d_{\vb g}$.
    By \Cref{lem:flat diracs} applied with $s = d_{\vb g}$, 
    $$
    (\lambda^{*})^{[t + d_{\vb g} + \frac d 2 ]} = \omega_1 \eval_{\xi_1}^{[t + d_{\vb g} + \frac d 2 ]}+ \dots + \omega_r \eval_{\xi_r}^{[t + d_{\vb g} + \frac d 2 ]}
    $$ 
    with $\omega_i>0$, $\Xi=\{\xi_1,\dots ,\xi_r \}\subset {\cS(\vb g)}$, $\ann_{t+1}(\lambda^*)=I(\Xi)_{t+1}$ and $(\ann_{t+1}(\lambda^*))=I(\Xi)$.
    {This implies in particular that $t+1 \ge \intdeg=\intdeg(\Xi)$, the interpolator degree of $\Xi$.}
    
    Let $\vb h = h_1, \dots , h_m\subset \ann_{t+1}(\lambda^*)$ be a {graded basis} of $I(\Xi)$ of degree $\le t+1$.
    As $(\lambda^*)^{[2(t+d_{\vb g})]}$ is generic in $\cL_d(\vb g)^{[2(t+d_{\vb g})]}$, for any $\lambda \in \cL_d(\vb g)$ we have $\ann_{t+d_{\vb g}}(\lambda^*)\subset \ann_{t+d_{\vb g}}(\lambda)$ by \Cref{prop::genericity}, and 
    $\braket{\lambda}{h_i^2}=0$.
    Then for any $p \in \R[\vb X]_{d_{\vb g} + \frac{d}{2}}$ we have $\braket{\lambda}{h_i p} = 0$, proving that $\lambda \in (\vb h)^{\perp}_{t+d_{\vb g}+\frac{d}{2}}=(\cI(\Xi)_{t+d_{\vb g}+\frac{d}{2}})^{\perp}$, i.e. $\cL_d(\vb g)^{[t+d_{\vb g}+\frac{d}{2}]} \subset (\cI(\Xi)_{t+d_{\vb g}+\frac{d}{2}})^{\perp}$. 
    
    {This also means that $\cL_d(\vb g)^{[t+d_{\vb g}+\frac{d}{2}]} \subset \cL_{[t+d_{\vb g}+\frac{d}{2}]}(\cI(\Xi)_{t+d_{\vb g}+\frac{d}{2}})$. As $t+d_{\vb g}+\frac{d}{2}\ge 2t+2 \ge 2\theta$}, we deduce from \Cref{prop:dual eval} that $\lambda^{[t + d_{\vb g} +\frac d 2 ]} \in \cone(\eval_{\xi_1},\ldots,\eval_{\xi_r})^{[t + d_{\vb g} + \frac d 2 ]}$. 
This shows that  
$
\cL_d(\vb g)^{[ t + d_{\vb g} + \frac d 2]} \subset \cone(\eval_{\xi_1},\ldots,\eval_{\xi_r})^{[t + d_{\vb g} + \frac d 2 ]}.
$
On the other hand the inclusion $\cL_d(\vb g)^{[ t + d_{\vb g}+ \frac d 2]} \supset \cone(\eval_{\xi_1},\ldots,\eval_{\xi_r})^{[t + d_{\vb g} + \frac d 2 ]}$ holds true since $\Xi \subset S$. Therefore $$
\cL_d(\vb g)^{[ t + d_{\vb g} + \frac d 2]} = \cone(\eval_{\xi_1},\ldots,\eval_{\xi_r})^{[t + d_{\vb g} + \frac d 2 ]}.
$$

Let us show  that $\Xi=S$. For $\zeta \in S$ we have $\eval_{\zeta}^{[t +  d_{\vb g} + \frac d 2]}\in \cL_d(\vb g)^{[t + d_{\vb g} +\frac d 2 ]}\subset (\vb h)^{\perp}_{t + d_{\vb g} +\frac d 2 }$, and thus
for $i=1, \ldots, m$, $\braket{\eval_{\zeta}}{h_i} = h_i(\zeta) = 0$. This shows that $\zeta$ is a root of $\vb h$ and thus $\zeta \in \Xi$.
We conclude that $\Xi=\{\xi_{1}, \ldots, \xi_{r}\}= S$.

The inclusion $\cI(S)_{2(t+d_{\vb g})} \subset \overline{\tqgen{g}{d}}$ follows from $\cL_d(\vb g)^{[t+d_{\vb g} +\frac{d}{2}]} \subset (\vb h)_{t+ d_{\vb g}+\frac{d}{2}}^{\perp}$. {Indeed,} $2(t+d_{\vb g}) \le t+d_{\vb g}+\frac{d}{2}$ and thus $\cL_d(\vb g)^{[2(t+d_{\vb g})]} \subset (\vb h)_{2(t+d_{\vb g})}^{\perp}$. Now notice that $(\cL_d(\vb g)^{[2(t+d_{\vb g})]})^{\vee} \subset \overline{\tqgen{g}{d}}$, using convex duality. Therefore dualizing $\cL_d(\vb g)^{[2(t+d_{\vb g})]} \subset (\vb h)_{{2(t+d_{\vb g})}}^{\perp}$ we obtain the desired inclusion. {Moreover,} $\cI(S)_{2(t+d_{\vb g})} \subset \overline{\tqgen{g}{d}} \cap - \overline{\tqgen{g}{d}} \subset \supp \widetilde{Q} = \sqrt[\R]{\supp Q}$, by \Cref{prop::qtilde}, and finally:
\[
    (\ann_{t+1}(\lambda^*)) = \cI(S) = (\cI(S)_{2(t+d_{\vb g})}) \subset \sqrt[\R]{\supp \cQ(\vb g)} \subset \sqrt[\R]{\supp \cO(\vb g)} = \cI(S),
\]
where the last equality is the Real Nullstellenstatz, see e.g. \cite[Note 2.2.2 (vi)]{marshall_positive_2008}. This shows that $(\ann_{t+1} (\lambda^*))= \sqrt[\R]{\supp \cQ(\vb g)} =\cI(S)$, concluding the proof.
\end{proof}

This theorem tells us that if the flat truncation property holds at degree $t\le \frac d 2 -d_{\vb g}$ for a generic element, then any linear functional in $\cL_d(\vb g)$, truncated in degree $t + \frac{d}{2} + d_{\vb g}$, coincides with a positive measure supported on $S=\{\xi_1,\dots ,\xi_r\}$.

In the following theorem we show that when $\supp(Q)$ is a
zero-dimensional ideal (and thus $S$ is finite), the flat truncation is satisfied {for all the positive linear functionals (and thus in particular for generic ones).}

\begin{theorem}
\label{thm::fin_set_qua}
    Suppose that $\dim \frac{\R[\vb X]}{\supp \cQ(\vb g)} = 0$. Then $S = \cS(\vb g)$ is finite and there exists $d\ge 2(\intdeg+d_{\vb g})$ such that $\cI(S)_{2(\intdeg+d_{\vb g})} \subset \supp \overline{\tqgen{g}{d}}$, where $\intdeg=\intdeg(S)$ is the interpolation degree of $S$, and for any $\lambda \in \cL_d(\vb g)$ the flat truncation property holds at degree $\intdeg$.
\end{theorem}
\begin{proof}
    Let $I=\supp \cQ(\vb g)$ and $J=\sqrt[\R]{\supp \cQ(\vb g)}$, and recall that $J = \sqrt{I}$, see e.g. \cite{marshall_positive_2008}. We deduce that $\dim \frac{\R[\vb X]}{J} = \dim \frac{\R[\vb X]}{I} = 0$ and by  \eqref{thm::marshall} we have
    $\cI(\cS(\vb g)) = \sqrt[\R]{\supp \cQ(\vb g)} = J$. Then
    $\cV_{\R}(J)=\cV_{\R}(\cI(\cS(\vb g))) = \cS(\vb g) =\{
    \xi_1,\dots ,\xi_r \}$ is finite.

    We choose a {graded basis} $\vb h$ of $J$ with $\deg \vb h \le \intdeg+1$. 
    By \Cref{cor::tmomgeo_support}, there exists $d\in \N$ such that $\cI(S)_{2(\intdeg+d_{\vb g})} \subset \supp \overline{\tqgen{g}{d}}$. From \Cref{cor::tmomgeo_support} and \Cref{prop:dual eval} we deduce that positive linear functionals in $\cL_d(\vb g)$ restricted to degree $\le 2(\intdeg+d_{\vb g})$ are conical sums of evaluations at $\xi_1,\dots,\xi_r$:
    \[
        \cL_{d}(\vb g)^{[2(\intdeg+d_{\vb g})]}\subset\cL_{2(\intdeg+d_{\vb g})}(\pm \vb h)=\cL_{2(\intdeg+d_{\vb g})}(J_{2(\intdeg+d_{\vb g})}) =\cone (\eval_{\xi_1},\dots,\eval_{\xi_r})^{[2(\intdeg+d_{\vb g})]},
    \]
    and for all $\lambda \in \cL_{d}(\vb g)$, we have $\rank H_{\lambda}^{\intdeg} = \rank H_{\lambda}^{ \intdeg+d_{\vb g}}$, {since
    \[
        \lambda^{[2(\intdeg+d_{\vb g})]} = \sum_{j=1}^{\rank H_\lambda^\intdeg} {\omega_{i_j} \, \eval_{\xi_{i_j}}^{[2(\intdeg+d_{\vb g})]}}
    \]
    with $\omega_{i_j} > 0$.}
\end{proof}
\Cref{thm::fin_set_qua} says that if $\dim \frac{\R[\vb X]}{\supp  \cQ(\vb g)} = 0$ then the minimal order for which we have flat truncation is not bigger than $d\ge 2(\intdeg+d_{\vb g})$ such that $\cI(S)_{2(\intdeg+d_{\vb g})} \subset {\supp}\, \overline{\tqgen{g}{d}}$. 
This degree is related to the minimal $d$ for which $\cI(S)=\sqrt[\R]{\supp \cQ(\vb g)}$
is generated by $\supp \overline{\tqgen{g}{d}}$, that is, 
the minimal degree $d$ such that $I(S)_{\intdeg+d_{\vb g}} \subset \ann_{\frac d 2}(\lambda^*)$ for a generic $\lambda^*\in \cL_d(\vb g)$. Moreover, as in the remark after \Cref{lem:flat diracs}, we can replace $\intdeg+d_{\vb g}$ with {$\intdeg+1$} if $d$ is big enough.
 
{\Cref{prop:flat truncation} and \Cref{thm::fin_set_qua}} show that if $\dim \frac{\R[\vb X]}{\supp \cQ(\vb g)} = 0$ then $S$ is a finite set of points and for a high enough degree $d$, all moment sequences in $\cL_d(\vb g)$, truncated in degree twice the interpolation degree are {represented by} a weighted sum of Dirac measures at these points.
In particular, it is possible to recover all the points in $S$ from a generic truncated moment sequence, see \cite{Henrion05detectingglobal}, 
\cite{abrilbucero:hal-00981546} and
\cite{mourrain_polynomialexponential_2018}.

{Results related to \Cref{thm::fin_set_qua} and \Cref{prop:flat truncation} were obtained in \cite{lasserre_semidefinite_2008} and
\cite{lasserre_moment_2013}, where they focus on the case of equations
$\vb h$ defining a finite real variety. They prove that, for $d$
big enough and for every positive linear functional $\lambda \in
\cL_{2d}(\vb \pm \vb h)$, the flat truncation property holds for $H_{\lambda}^d$, and that
$\lambda^{[2d]}$ is a conic combination of evaluations at the points of $\cV_{\R}(\vb h)$. This can be deduced from
\Cref{thm::fin_set_qua}, since in the case where $\cV_{\R}(\vb h)=\{\xi_1,\dots, \xi_r\}$ is non-empty and finite, $\dim \frac{\R[\vb X]}{\supp \cQ(\pm \vb h)}=0$.  \Cref{prop:flat truncation} generalizes \cite[Prop.~4.5]{lasserre_semidefinite_2008}, since our result does not assume any explicit equality constraints, emphasizes the role of the support of the quadratic module and the interpolation degree, and provides better degree bounds for the flat truncation degree. \Cref{thm::fin_set_qua} analogously generalizes \cite[Prop.~4.6]{lasserre_semidefinite_2008}. These generalizations are crucial for the characterization of flat truncation in \Cref{thm::flat_iff}, and for \Cref{cor::generic_exactness}.}

{In \cite[Rem.~4.9]{lasserre_semidefinite_2008} the authors also mention that their results can be proved for a preordering defining a finite semialgebraic set. This result can be deduced directly from \Cref{thm::fin_set_qua}, since when $S = S(\vb g)=\{\xi_1,\dots
   ,\xi_r\}$ is non-empty and finite, we have by the Real Nullstellensatz , $\dim \frac{\R[\vb X]}{\supp \cO(\vb g)} =\dim \frac{\R[\vb X]}{\sqrt[\R]{\supp \cO(\vb g)}}= \dim \frac{\R[\vb X]}{\cI(\cS(\vb g))} = 0$.}

But \Cref{thm::fin_set_qua} is even more general, as shown
by the following example of a quadratic module, whose support is {zero-dimensional}, but it is not a preordering.
\begin{example}[{\cite[Ex.~ 7.4.5~(1)]{marshall_positive_2008}}]
    \label{ex::supp_non_pre}
    Let $Q = \cQ(X,Y,1-X,1-Y, -X^4, -Y^4) \subset \R[X, Y]$. In this
    case $\supp Q$, which contains $X^4$ and $Y^4$, is zero-dimensional and $Q$ is not a
    preordering  since $XY \notin Q$. \Cref{thm::fin_set_qua} applies in this case, but the results cannot be deduced from \cite{lasserre_semidefinite_2008} or
\cite{lasserre_moment_2013}.
\end{example}

As we will see, in Polynomial Optimization {Problems}, flat truncation implies MoM exactness and thus finite convergence. Moreover, it allows extracting the minimizers from an optimal sequence.
\section{Flat truncation in  polynomial optimization problems}
\label{sec::flat_tru}
In this section, we analyze when flat truncation occurs in the  Polynomial Optimization Problem, which consists of
minimizing $f\in \R[\vb X]$ on the basic semialgebraic set $S=\cS(\vb
g)$ where $\vb g = \{ \, g_{1}, \ldots, g_{s} \, \}$ are finitely many real polynomials. 

Recall that we denote $f^*$ the minimum of $f$ on $S$. In the following, we will consider the semialgebraic set $S^{\min} = \cS(\vb g, \pm (f - f^*)) = \cS(\vb g) \cap \{ x \in \R^n \mid f(x) = f^* \}$ and assume that it is nonempty. 

{The interest in flat truncation properties in polynomial optimization arises from two aspects. First, in such a case we can certify finite convergence and exactness of the MoM hierarchy, see \Cref{thm::flat_implies_exact}. Second, if we have a flat moment matrix we can recover the global minimizers of the problem. For more details on the algorithm to extract the minimizers, we refer the reader to \cite{Henrion05detectingglobal, abrilbucero:hal-00981546, mourrain_polynomialexponential_2018}}.

In the following, we adapt the results of \Cref{sec:Geometry} to the quadratic module $\cQ(\vb g, \pm (f - f^*))$
defining the minimizers $S^{\min}$.

\subsection{Flat truncation degree}\label{sec:flat truncation degree}
{In the following, we aim at analyzing the degree at which flat truncation holds using the interpolation degree, and we provide the first necessary and sufficient condition for flat truncation (\Cref{thm::flat_iff}). But first, we describe the consequences of flat truncation for generic minimizing linear functionals in \Cref{thm::flat_implies_exact}. This theorem extends related results, e.g. in \cite{laurent_sums_2009, nie_certifying_2013}: we perform a detailed comparison with the existing literature after the proof.}
\begin{theorem}
\label{thm::flat_implies_exact}
Consider the problem of minimizing $f$ on $\cS(\vb g)$. If the flat truncation property holds for a generic $\lambda^*\in \cL_{2d}^{\min}(f;\vb g)$ at a degree $t$ such that $\deg(f)-d_{\vb g}-d \le t \le  d- d_{\vb g}$, then:
\begin{enumerate}
    \item $f^* = f^*_{\mom, d}$ (i.e. we have MoM finite convergence);
    \item the set of minimizers $S^{\min}=\{\xi_1,\dots ,\xi_r\}$ is non-empty and finite;
    \item $\ker H_{\lambda^*}^{t+1}=\ann_{t+1}(\lambda^*) = \cI(S^{\min})_{t+1}$ (i.e. the kernel of the truncated moment matrix equals the truncated ideal of the minimizers) and $\cV(\ann_{t+1}(\lambda^*)) = S^{\min}$;
    \item $\displaystyle \cL_{2d}^{\min}(f; \vb g)^{[t + d_{\vb g} +  d  ]} = \cone (\eval_{\xi_1},\dots,\eval_{\xi_r})^{[t + d_{\vb g} + d ]}$ (i.e. all the minimizing truncated feasible moment sequences are conic sums of evaluations at the minimizers) {and all $\lambda \in \cL_{2d}^{\min}(f;\vb g)$ have flat truncation at degree $t$};
    \item the MoM hierarchy is exact.
\end{enumerate}
\end{theorem}
\begin{proof}
Let $\lambda^*\in \cL_{2d}^{\min}(f; \vb g)$ be generic such that $\rank H^{t}_{\lambda^*}= \rank H^{t+d_{\vb g
}}_{\lambda^*}$ with 
$\deg(f)\le t + d_{\vb g}+d$ and $t + d_{\vb g} \le d$. Then by \Cref{lem:flat diracs}, $(\lambda^*)^{[t+ d_{\vb g} + d]}=\sum_{i=1}^{r} \omega_i \eval_{\xi_i}^{[t + d_{\vb g} + d]}$ with $\xi_i\in S = \cS(\vb g)$, $\omega_i > 0$, $\ann_{t+ 1}(\lambda^*) = \cI(\xi_1, \dots \xi_r)_{t+ 1} = \cI(\Xi)_{t+ 1} $ and $\cV(\ann_{t+1}(\lambda^*)) = \Xi$. Notice that $f(\xi_i) \ge f^*$ since $\xi_i \in S$. 

We show now that $S^{\min} = \Xi$.
As $\braket{\lambda^*}{1} = 1$ we have $\sum_{i=1}^{r} \omega_i = 1$. Moreover $f^*_{\mom, d} = \braket{\lambda^*}{f} \le f^*$ and since $\deg(f)\le t + d_{\vb g}+d$ we obtain:
$$
f^* \ge \braket{\lambda^*}{f} = \braket{(\lambda^*)^{[t+ d_{\vb g} +d]}}{f} = \sum_{i=1}^{r} \omega_i \braket{\eval_{\xi_i}^{[t + d_{\vb g} + d]}}{f}=  \sum_{i=1}^{r} \omega_i f(\xi_i) \ge f^*. 
$$
This implies that $f(\xi_i)=f^*$ for $i=1,\ldots, r$. Therefore $f^* = f^*_{\mom, d}$ and $S^{\min} \supset \Xi$.

From \Cref{prop::genericity} we have that $\lambda^*\in \cL_{2d}^{\min} (f; \vb g)$ {being} generic implies that $(\lambda^*)^{[2(t+ d_{\vb g})]}$ is generic in $\cL_{2d}^{\min} (f; \vb g)^{[2(t+ d_{\vb g})]}$. Moreover $(\lambda^*)^{[2(t+ d_{\vb g})]} = \sum_{i=1}^{r} \omega_i \eval_{\xi_i}^{[2(t+ d_{\vb g})]} \in \cLone_{2d}(\vb g, \pm (f-f^*))^{[2(t+ d_{\vb g})]}$ since $\Xi \subset S^{\min} = \cS(\vb g, \pm(f-f^*))$. Then, as $\cLone_{2d}(\vb g, \pm (f-f^*)) \subset \cL_{2d}^{\min}(f; \vb g)$ and $(\lambda^*)^{[2(t+ d_{\vb g})]}$ is generic in $\cL_{2d}^{\min} (f; \vb g)^{[2(t+ d_{\vb g})]}$, we have
\[
   \forall \lambda \in \cL_{2d}(\vb g, \pm (f-f^*)) \quad \ann_{t+ d_{\vb g}}(\lambda^*) \subset \ann_{t+ d_{\vb g}}(\lambda), 
\]
i.e. $(\lambda^*)^{[2(t+ d_{\vb g})]}$ is generic in $\cL_{2d}(\vb g, \pm (f-f^*))^{[2(t+ d_{\vb g})]}$. 
We can then conclude from \Cref{prop:flat truncation} that:
\begin{itemize}
    \item $S^{\min} = \Xi$;
    \item $\cL_{2d}^{\min}(f;\vb g)^{[t + d_{\vb g} +  d  ]} = \cone (\eval_{\xi_1},\dots,\eval_{\xi_r})^{[t + d_{\vb g} + d ]}$;
    \item {$t \ge \theta(\xi_1, \dots , \xi_r)$;}
    \item {$\supp \cQ(\vb g, \pm (f-f^*))$ is zero dimensional, {since $\sqrt{\supp \cQ(\vb g, \pm (f-f^*))}= \sqrt[\R]{\supp \cQ(\vb g, \pm (f-f^*))}$ is zero-dimensional.}}
\end{itemize}
{From \Cref{thm::fin_set_qua}, applied to $\cQ(\vb g, \pm (f-f^*))$, we deduce that all $\lambda \in \cL_{2d}^{\min}(f;\vb g)$ have flat truncation at degree $t$}.

Finally we show MoM exactness. For every $d'\ge d$ and $\lambda \in \cL_{2d'}^{\min}(f; \vb g)$, we have $\lambda^{[2d]} \in \cL_{2d}^{\min}(f; \vb g)$ since $\braket{\lambda}{f}=f^*$. 
Therefore $\lambda$ has flat truncation in degree $t$ and by \Cref{lem:flat diracs}, $\lambda^{[t + d_{\vb g}+ d']}$ is 
 {represented by} a convex sum of Dirac measures at points in $S$ (that are the minimizers $\xi_1,\dots ,\xi_r$). This shows that the moment hierarchy is exact, since increasing $d'$ we increase also the truncation degree where $\lambda$ coincides with a weighted sum of evaluations at the minimizers, {i.e., setting $k = t + d_{\vb g}+ d'$, we have $\cL_{2d'}^{\min}(f; \vb g)^{[k]} \subset \cMone(\cS(\vb g))^{[k]}$ for all $d'$. As increasing $d'$ in the previous inclusion we also increase $k$, this proves MoM exactness and concludes the proof.}
\end{proof}

{\Cref{thm::flat_implies_exact} relaxes previous degree conditions. In \cite[Th.~6.18]{laurent_sums_2009}, the degree condition to deduce $(i)$ and $(ii)$ in \Cref{thm::flat_implies_exact} is $\deg(f) \le 2 t + 2 d_{\vb g}$, which is in general a stronger condition than $\deg(f) \le  t + d_{\vb g}+d$ in \Cref{thm::flat_implies_exact}.
\Cref{thm::flat_implies_exact} also shows that the kernel of the moment matrix of a generic truncated moment sequence, $\ann_{t+1}(\lambda^*)$, is the truncated vanishing ideal of the minimizers and that the hierarchy is exact. This means that any element in $\cL_{2d}^{\min}(f; \vb g)$ truncated in any degree $t$ is {represented by} a measure, provided $d\ge t$ is big enough.
Recall from \Cref{sec:Geometry} that the points in the relative interior of $\cL^{\min}_{2d}(f; \vb g)$ are generic in this convex cone, and thus interior point SDP solvers in practice return approximations of generic linear functionals in $\cL_{2d}^{\min}(f; \vb g)$. \Cref{thm::flat_implies_exact} shows that, if \emph{any} generic minimizing linear functional has flat truncation at degree $t$, {then} \emph{all} minimizing linear functional have flat truncation at degree $t$. Similar results have been described in \cite{lasserre_semidefinite_2008, lasserre_moment_2013}, but only for the feasibility problem in the zero-dimensional case. 
This result was not previously described in the general polynomial optimization case (e.g. this statement cannot be found in \cite[Th.~6.18]{laurent_sums_2009}). In \cite{nie_certifying_2013}, it is shown that finite convergence of the hierarchies and flat truncation for all minimizing linear functionals are generically equivalent conditions, but the special properties of generic minimizing linear functionals are not considered (see also the discussion after \Cref{thm::flat_iff}).}

A key ingredient in this analysis is \Cref{lem:flat diracs}. From \Cref{lem:flat diracs} and the remark after it, if $d$ is big enough the results of \Cref{thm::flat_implies_exact} hold true, if we replace the condition $\rank H_{\lambda^*}^{t} = \rank H_{\lambda^*}^{t+d_{\vb g}}$ with $\rank H_{\lambda^*}^{t}= \rank H_{\lambda^*}^{t+1}$.

However, we show in \Cref{ex::flat_out} that the condition $\rank H_{\lambda^*}^{t}= \rank H_{\lambda^*}^{t+1}$ is in general not sufficient to conclude that the points extracted from the moment matrix are inside the semialgebraic set. 
This is the first example where such a pathological behaviour is explicit.

\begin{example}
\label{ex::flat_out}
We consider the problem of minimizing $f = (1+X)(X-1)^2$ on $\cS(1-X^2, -X^3) = [-1, 0]$. Notice that the SoS hierarchy is exact, since $f^* = 0$ and:
\[
    (1+X)(X-1)^2 = \frac{1}{2}\big( (1+X)^2 + 1 - X^2 \big)(X-1)^2 \in \cQ_4(1-X^2, -X^3).
\]
This implies that $f^*_{\sos, 2} = f^*_{\mom, 2} = f^*$.
The only minimizer of $f$ on $S$ is $-1$, and $\cI(-1) = (X+1)$: therefore we would expect to get flat truncation at degree zero for a generic element, and in particular $\rank H_{\lambda^*}^0  = \rank H_{\lambda^*}^1 = 1$. But this is not the case if we consider the MoM relaxation of order $2$. Indeed an explicit computation shows that $\lambda = \frac{1}{2}(\eval_{-1}^{[4]} + \eval_{1}^{[4]}) \in \cL_4^{\min} (f; \vb g)$, and $\rank H_{\lambda}^1 = \rank H_{\lambda}^2 = 2$. Therefore a generic $\lambda^* \in \cL_4^{\min} (f; \vb g)$ cannot satisfy flat truncation at degree $t=0$. More precisely, it is possible to show that $\cL_4^{\min} (f; \vb g)= \conv \big(\eval_{-1}^{[4]}, \frac{1}{2}(\eval_{-1}^{[4]} + \eval_{1}^{[4]})\big)$. Therefore a generic $\lambda^* \in \cL_4^{\min} (f; \vb g)$ will also satisfy $\rank H_{\lambda}^1 = \rank H_{\lambda}^2 = 2$.

We confirm numerically the computation above, using the package \texttt{MomentPolynomialOpt.jl} to compute $f^*$ and a generic $\lambda^* \in \cL^{\min}_4(f; \vb g)$: the pseudo-moments that we obtain are
    \begin{table}[h!]
    \begin{tabularx}{\textwidth}{X X X}
$\lambda^*_0 = 0.9999999989784975,$ &$\lambda^*_1 = - 0.3530324749675295$ & $\lambda^*_2 = 0.9998474115299072$ \\
$\lambda^*_3 = - 0.3531851571450224$ &$\lambda^*_4 = 0.9996947364721432.$ & 
    \end{tabularx}
    \end{table}
We compute the singular values of $H_{\lambda^*}^0 $, $H_{\lambda^*}^1$ and $H_{\lambda^*}^2$ to have a numerically stable indication of the ranks:
\begin{align*}
    \text{Sing. Val. of } H_{\lambda^*}^0&: 0.9999999989784975 \\
    \text{Sing. Val. of } H_{\lambda^*}^1&: 1.352956188465637,  0.6468912220427679 \\
    \text{Sing. Val. of } H_{\lambda^*}^2&: 2.2063794508570065, 0.7931627759613444, 7.983780245045715 \cdot 10^{-8}
\end{align*}
This confirms the theoretical description and shows that the rank condition is numerically satisfied for $t=1$. The points extracted from the matrix are $\xi_1 \approx 0.9997640487211856$ and $\xi_2 \approx -1.0000000483192044$: notice that $\xi_1 \notin S$. This happens because the condition $\rank H^t_{\lambda^*} = \rank H^{t+d_{\vb g}}_{\lambda^*}$ is not satisfied (we cannot compute $H^{t+d_{\vb g}}_{\lambda^*} = H^{3}_{\lambda^*}$ as $3 = t+d_{\vb g} > d =2$).
 
On the other hand, if we increase the order of the relaxation and compute $\lambda^* \in \cL^{\min}_6(f; \vb g)$ generic, we can verify flat truncation for $t = 0$ and the only point extracted is $-1$. Moreover, notice, from \Cref{lem:flat diracs} applied with $s=1$ and the remark below, that in this case it is enough to check $ \rank H^0_{\lambda^*} = \rank H^1_{\lambda^*}$ to verify that $\rank H^t_{\lambda^*} = \rank H^{t+d_{\vb g}}_{\lambda^*}$, since the condition $0 = t \le d + s - \deg(\vb g) =1$ is satisfied.
\end{example}
We have seen that flat truncation implies MoM exactness and a finite set of minimizers.
We now show that, under the assumption of MoM finite convergence, flat truncation is equivalent to {the zero-dimensionality of the} support for the quadratic module $Q+(f-f^*)$ defining the minimizers.

We first need a technical lemma, that will be important to investigate the relationship between $\cL_{2d}^{\min}(f; \vb g)$ and $\cLone_{2d}(\vb g, \pm(f-f^*))$. Indeed, notice that $\cLone_{2d}(\vb g, \pm(f-f^*)) \subset \cL_{2d}^{\min}(f; \vb g)$, by definition, but the converse inclusion is not true in general, since for $\lambda \in \cL_{2d}^{\min}(f; \vb g)$ we only have $\braket{\lambda}{f}= f^*$, and not $f-f^* \in \ann_{d - \frac{\deg(f)}{2}}(\lambda)$.

\begin{lemma}
\label{lem::zero_implies_ann}
    Let $f \in \tqgen{g}{2k}$, $\lambda \in \cL_{2d}(\vb g)$ and $t\in \N$ with $0\le t \le d-k$. Then $\braket{\lambda}{f} = 0$ implies for all $q \in \R[X]_t$, $\braket{\lambda}{qf} = 0$. In other words, $f \in \ann_t(\lambda)$.
\end{lemma}
\begin{proof}
    {We set $g_0 = 1$ for notational convenience. Let $f = \sum_{i} s_i g_i = \sum_{i, j} p_{i,j}^2 g_i \in \tqgen{g}{2k}$, with $\deg p_{i,j}^2 g_i \le 2k$.
    We want to prove that for all $q \in \R[\vb X]$ such that $\deg(q) \le t$ we have $\braket{\lambda}{qf} = 0$. In particular, it is enough to prove that:
    \begin{equation}
        \label{eq::zero}
        \braket{\lambda}{q p_{i,j}^2 g_i}=0 \text{ for all } i,j \text{ and } q \in \R[\vb X]_t.
    \end{equation}}
    
    {Now, notice that $\braket{\lambda}{f} = 0$ implies $\braket{\lambda}{p_{i,j}^2g_i} = 0$ for all $i,j$ (since $\lambda$ is nonnegative on $\tqgen{g}{2k}$), and therefore for all $T \in \R$ and $h \in \R[\vb X]_{t + \deg p_{i,j}}$ we have:
    \begin{equation*}
        0\le\braket{\lambda}{(p_{i,j} - Th)^2 g_i} = T^2 \braket{\lambda}{h^2g_i} - 2T\braket{\lambda}{h p_{i,j}g_i}    
    \end{equation*}
    (we can apply $\lambda$ to $(p_{i,j} - Th)^2 g_i$ since $\deg((p_{i,j} - Th)^2 g_i) \le 2t+2k \le 2d$). For all $h \in \R[\vb X]_{t + \deg p_{i,j}}$ the polynomial $T^2 \braket{\lambda}{h^2g_i} - 2T\braket{\lambda}{h p_{i,j}g_i}$ is nonnegative by the above inequality, and since it vanishes at $T=0$, it has a double root at $T=0$. This implies that $\braket{\lambda}{h p_{i,j}g_i}=0$ for all $h \in \R[\vb X]_{t + \deg p_{i,j}}$. If we substitute $h = q p_{i,j}$, we deduce \cref{eq::zero} for all $q \in \R[\vb X]_t$, and thus $f \in \ann_t(\lambda)$.}
\end{proof}

We can now prove the equivalence between the flat truncation and {the zero-dimensionality of} the support for the quadratic module $Q + (f-f^*)$ defining the minimizers. The proof relies on two main ideas:
\begin{enumerate}
\item we use \Cref{lem::zero_implies_ann} to move from $\cL_{2d}^{\min}(f; \vb g)$ to $\cLone_{2d}(\vb g, \pm (f-f^*))$;
\item we use \Cref{thm::fin_set_qua} to analyze the zero dimensional case.
\end{enumerate}
\begin{theorem}\label{thm::flat_iff}
    Assume that we have MoM finite convergence. Then $\dim \frac{\R[\vb X]}{\supp (\cQ(\vb g) + (f-f^*))} = 0$ if and only if there exists $d$ such that a generic $\lambda^* \in \cL_{2d}^{\min}(f; \vb g)$ has flat truncation. 
    
    In this case, if $\intdeg = \intdeg (S^{\min})$ is the interpolation degree and $D = \max(d_{\vb g}, \ceil{\frac{\deg(f)}{2}})$, there exists $\delta\in \N$ such that $f - f^* \in \overline{\tqgen{g}{2\delta}}$ and  flat truncation at degree $\intdeg$ happens for {all} $\lambda \in \cL_{2d}^{\min}(f;\vb g)$, when $d$ is such that:
    \begin{enumerate}
    \item $\displaystyle (\sqrt[\R]{\supp Q(\vb g)})_{2\delta+2\intdeg+2D-\deg(f)} \subset \overline{\tqgen{g}{2d}}$;
        \item $\displaystyle \cI(S^{\min})_{2\intdeg+2D} \subset \overline{\tqgen{g}{2d}+ (f-f^*)_{2d}}$;
        \item $\displaystyle \delta + 2\intdeg + 2D - \deg(f) \le d$.
    \end{enumerate}
\end{theorem}
\begin{proof}
    Let us assume without loss of generality that $f^* = 0$.

    We first show that flat truncation implies  $\dim \frac{\R[\vb X]}{\supp (Q + (f))} = 0$. As in the proof of \Cref{thm::flat_implies_exact}, if $\lambda^* \in \cL_{2d}^{\min}(f; \vb g)$ is generic satisfying flat truncation at degree $t$ then $(\lambda^*)^{[2(t+d_{\vb g})]}$ is a generic element of $\cL_{2d}(\vb g, \pm f)^{[2(t+d_{\vb g})]}$. Since the flat truncation property is satisfied, we conclude from \Cref{prop:flat truncation}, {in particular point (iv)}, that $\sqrt[\R]{\supp (Q + (f))} = (\ann_{t+1}(\lambda^*)) = \cI(S^{\min})$ and $\dim \frac{\R[\vb X]}{\supp (Q + (f))} = \dim \frac{\R[\vb X]}{\cI( S^{\min})} = 0$.
    
    Conversely, if $\dim \frac{\R[\vb X]}{\supp (Q + (f))} = 0$, we deduce from \Cref{thm::fin_set_qua} that the flat truncation property is satisfied for any $\lambda \in \cL_{2d}(\vb g, \pm f)$ at degree $\intdeg=\intdeg(S^{\min})=\intdeg(S(\vb g , \pm f)$ for $d$ such that $\cI(S^{\min})_{2(\intdeg+D)} \subset \overline{\tqgen{g}{2d}+ (f)_{2d}}$. 
     Let $a = 2\intdeg+2D$ and $\lambda \in \cL_{2d}^{\min}(f;\vb g)$ generic. We want to show that $\lambda^{[a]} \in \cL_{2d}(\vb g, \pm f)^{[a]}$, so that we can conclude using \Cref{thm::fin_set_qua}.
    Since $\lambda \in \cL_{2d}^{\min}(f;\vb g) \subset \cL_{2d}(\vb g)$, it is sufficient to prove that: 
    \begin{equation}
        \label{eq::claim}
        \braket{\lambda}{q f} = 0\text{ for all $q$ of degree $\le a - \deg (f)$}.
    \end{equation}
    
    We prove now \eqref{eq::claim}, starting from $\braket{\lambda}{f} = f^*=0$. MoM finite convergence implies that $\braket{\lambda}{f} \ge 0$ for all $\lambda \in \cL_{2d}(\vb g)$, and therefore $f \in \cL_{2d}(\vb g)^{\vee} = \overline{\tqgen{g}{2d}}$. Let $\delta \le d$ be minimal such that $f \in \overline{\tqgen{g}{2\delta}}$ and let $\vb h = h_1, \dots, h_m$ be a graded basis of $\sqrt[\R]{\supp Q}$. From \cite[Lem. 4.1.4]{marshall_positive_2008} we deduce that $\tqgen{g}{2\delta} + (\vb h)_{2\delta}$ is closed (as a subset of $\R[\vb X]_{2\delta}$ with the Euclidean topology, see also the proof of \Cref{prop:radical closed}), and therefore $\overline{\tqgen{g}{2\delta}} \subset \tqgen{g}{2\delta} + (\vb h)_{2\delta}$. Thus we can write:
    $$
    f = g + h = \sum_{i=0}^s s_i g_i + \sum_{i=1}^m p_i h_i \in \tqgen{g}{2\delta} + (\vb h)_{2\delta},
    $$ 
    where we set $g_0 = 1$ for notation convenience, $g = \sum_{i=0}^s s_i g_i \in \tqgen{g}{2\delta}$ and $h = \sum_{i=1}^m p_i h_i \in (\vb h)_{2\delta}$. It is then enough to prove that $\braket{\lambda}{qg} = \braket{\lambda}{qh} = 0$, where $\lambda \in \cL_{2d}^{\min}(f;\vb g)$ and $\deg(qg)\le b, \deg(qh) \le b $ for $b =2\delta+ a-\deg(f)= 2\delta+ 2\intdeg + 2D -\deg(f)$.
    
    We start by proving $\braket{\lambda}{qh} = 0$. We deduce from \Cref{lem::dual_support} that for $d$ big enough we have $(\vb h)_{b} \subset \overline{\tqgen{g}{2d}}$ and $\cL_{2d}(\vb g)^{[b]} \subset \cL_{b}(\pm \vb h)$. Therefore $$\braket{\lambda}{qh} = \braket{\lambda^{[b]}}{qh} = 0.$$
    
    Now we prove that $\braket{\lambda}{qg} = 0$. Since $\delta + (a-\deg(f)) \le d$, we can apply \Cref{lem::zero_implies_ann} with $g \in \tqgen{g}{2\delta}$ and $t = a - \deg(f) \ge \deg(q)$, and conclude that $\braket{\lambda}{qg} = 0$, as desired.
    
    Therefore $\braket{\lambda}{q f} =\braket{\lambda^*}{qg} + \braket{\lambda^*}{qh} = 0$ for all $q$ of degree $\le a - \deg (f)$ and \eqref{eq::claim} is satisfied. This implies that $\lambda^{[a]} \in \cL_{2d}(\vb g, \pm f)^{[a]}$, {or in other words $\cL_{2d}^{\min}(f;\vb g)^{[a]} \subset \cL_{2d}(\vb g, \pm f)^{[a]}$}. 
    Therefore we can apply \Cref{thm::fin_set_qua} to conclude that the flat truncation property is satisfied for {all} $\lambda \in \cL_{2d}^{\min}(f;\vb g)$ at degree $\intdeg$.
\end{proof}

Let us briefly comment the degree conditions in \Cref{thm::flat_iff}.
\begin{enumerate}
    \item If $S$ has nonempty interior, it is not necessary to check the first condition, since in this case $\supp Q = 0$. More generally if the quadratic module is reduced, that is if $\sqrt[\R]{\supp Q} = \supp Q$, the first condition is automatically satisfied;
    \item The second condition is the key one:  it tells us that flat truncation happens when the ideal of the minimizers, truncated in the appropriate degree, can be described using the truncated quadratic module and the truncated ideal generated by $f-f^*$;
    \item The third condition is technical, derived from \Cref{lem::zero_implies_ann}. It allows to move from $\cL^{\min}_{2d}(f;\vb g)$ to $\cL_{2d}(\vb g, \pm(f-f^*))$, where we can apply the results of the previous section.
\end{enumerate}

{Related properties have been previously investigated. It is shown in \cite[Th.~2.2]{nie_certifying_2013} that, under genericity assumptions, if for an order $d$ big enough we have $f^*_{\sos, d} = f^*_{\mom, d}$ (strong duality) and $\sup = \max$ in the definition of $f^*_{\sos, d}$, then there is finite convergence (that is $f^*_{\mom, d} = f^*$) if and only if flat truncation is satisfied for every $\lambda \in \cL_{2d}^{\min}(f;\vb g)$. \Cref{thm::flat_iff} applies for other cases, for instance when there is finite convergence but the SoS hierachy is not exact (see \Cref{ex::fin_var}). This is made possible by analysing the closure of the quadratic modules we are considering.
As a corollary of \Cref{thm::flat_iff} we will show (in \Cref{thm::BHC} and \Cref{cor::generic_exactness}) that, under genericity assumption,
we have finite convergence, that the MoM hierarchy is exact and that the flat truncation property is satisfied.}

{Another improvement made is the estimation of the order $d$ of the hierarchy that is sufficient to have flat truncation, answering a question in \cite[Sec.~5]{nie_certifying_2013}. 
This is the first result in this direction. These conditions depends on properties of the minimizers and the quadratic module $\cQ_{2d}(\vb g)$ that might be difficult to check a priori. In particular, they depend on the interpolation degree of the minimizers. However, they may be analyzed in some specific cases, such as optimization problems with a single minimizer, to deduce more precise bounds. Moreover, the relation between the degree of flat truncation and the interpolation degree $\intdeg$ shows that computing the degree of flat truncation is at least as hard as computing $\intdeg$, which is difficult with no a priori knowledge on the minimizers.}

We illustrate \Cref{thm::flat_iff} in the following example, showing how it can help to predict the flat truncation degree.
\begin{example}\label{ex::BHC2}
    We continue \Cref{ex::BHC}. Notice that  $f - f^* = X^2 \in Q_2 := \cQ_2(\vb g) = \cQ_2(1-X^2-Y^2, X+Y-1)$ (i.e. the SoS hierarchy is exact) and then the MoM hierarchy has finite convergence.
    Using \Cref{thm::flat_iff}, we analyse if flat truncation holds at some degree. We have $\cI(S^{\min}) = (X, Y-1) \subset \sqrt[\R]{\supp(Q + (f-f^*))} = \sqrt[\R]{\supp(Q + (X^2))}$ where $Q := \cQ(1-X^2-Y^2, X+Y-1)$.
    Indeed:
    \begin{align*}
        X &= \frac{X^2+(Y-1)^2}{2} + \frac{1-X^2-Y^2}{2} + X + Y -1 \in Q_2 \subset \overline{Q_2+(X^2)_2}\\
        - X + \epsilon &= \frac{\epsilon}{2}\big( 1 - \frac{X^2}{\epsilon^2} + (1-\frac{X}{\epsilon})^2 \big) \in Q_2 + (X^2)_2 \ \forall \epsilon>0 \Rightarrow -X \in \overline{Q_2 + (X^2)_2} \\
        1-Y &= \frac{1}{2} \big(X^2 + (1-Y)^2 + 1 - X^2 - Y^2 \big) \in Q_2 \subset \overline{Q_2 + (X^2)_2}\\
        Y -1 &= X+Y-1 - X \in Q_2 + \overline{Q_2+(X^2)_2} = \overline{Q_2+(X^2)_2}
    \end{align*}
    that implies $(X, Y-1)_1 \subset \supp( \overline{Q_2+(X^2)_2}) \subset \sqrt[\R]{\supp(Q + (f-f^*))} $  and thus $\dim \frac{\R[\vb X]}{\supp(Q + (X^2))} = 0$. \Cref{thm::flat_iff} implies that flat truncation holds for a  high enough order $d$ of the MoM relaxation. 
    
    We investigate the degree conditions in \Cref{thm::flat_iff} to prove that flat truncation happens for the MoM hierarchy at order $d=1$. We have $I(S^{\min})=(X, Y-1)$, $\intdeg=0$, $d_{\vb g}=1$, $\deg(f) = 2$, $D = 1$ and $\delta=1$.
    \begin{enumerate}
        \item As $S$ has nonempty interior, $\supp Q = 0$ and the first point (i) is satisfied.
        \item Notice that $2(\intdeg + D) = 2$, and therefore we have to show that $(X, Y-1)_2 \subset \overline{Q_2+(X^2)_2}$. Since we have shown above that $(X, Y-1)_1 \subset \overline{Q_2+(X^2)_2}$, it is enough to prove that $\pm X^2, \pm X(Y-1), \pm (Y-1)^2 \in \overline{Q_2+(X^2)_2}$. Now, $\pm X^2, \, (Y-1)^2 \in \overline{Q_2+(X^2)_2}$ by definition. Finally:
        \begin{align*}
            -(Y-1)^2 &= 1-Y^2 - X^2 + X^2 + 2 (X+Y-1) -2X \in Q_2+ \overline{Q_2+(X^2)_2} = \overline{Q_2+(X^2)_2} \\
            \pm X(Y-1)&= \frac{1}{2}\big( (\pm X +(Y-1))^2 - X^2 - (Y-1)^2 \big) \in \overline{Q_2+(X^2)_2},
        \end{align*}
        concluding the proof of the second point (ii).
        \item We have $1 =  \delta + 2\intdeg +2D - \deg(f) \le d = 1$, and thus the third point (iii) is satisfied.
    \end{enumerate}
    Therefore flat truncation happens at degree $\intdeg=0$ for the MoM hierarchy at order $d=1$.
\end{example}

\subsection{Boundary Hessian Conditions}\label{sec::BHC}

In this section, we show that if
regularity conditions, known as Boundary Hessian
Conditions (BHC), are satisfied, then the flat truncation property holds. These are conditions on the minimizers of a polynomial $f$ on a basic semialgebraic set $S$ introduced by
    Marshall in \cite{MarshallRepresentationsnonnegativepolynomials2006}
    and \cite{MarshallRepresentationsNonNegativePolynomials2009},
    which are particular cases of the so
    called local-global principle. Under these conditions, global
    properties of polynomials (e.g. $f\in Q$) can be deduced from
    local properties (e.g. checking the BHC at the minimizers of $f$ on
    $\cS(Q)$). We refer to
    \cite{scheiderer_distinguished_2005}, \cite{scheiderer_sums_2006} and
    \cite[Ch. 9]{marshall_positive_2008} for more details.
    We introduce BHC conditions following \cite{MarshallRepresentationsnonnegativepolynomials2006,nie_optimality_2014}.

    \begin{definition}[Boundary Hessian Conditions]\label{def:BHC}
Consider a POP with inequality constraints $\vb g= g_1, \dots , g_r$, equality constraints $\vb h = h_1, \dots , h_s$ and objective function $f$. Let $V = {\cV_\R(\vb h)}  \subset \R^n$ and suppose that $Q = \cQ(\vb g,\pm \vb h )$ is Archimedean.
We say that the \emph{Boundary Hessian Conditions} hold at a minimizer point $\xi \in S(\vb g, \pm \vb h)$ of $f$ if $\xi$ is a smooth point of $V$ and:
        \begin{enumerate}
            \item we can choose $g_{i_1} = t_1, \dots, g_{i_k}= t_k$ that are part of a regular system of parameters $t_1,\dots,t_m$, $m \ge k$, for $V$ at $\xi$\footnote{{this means that the $t_i$'s vanish at $\xi$ and that their gradients (or differentials) are linearly independent at $\xi$}} and for some neighbourhood $U$ of $\xi$ we have $\cS(g_{i_1}, \dots, g_{i_k}, \pm \vb h) \cap U = \cS(\vb g , \pm \vb h) \cap U$;
            \item On $V$, locally at $\xi$ we have that $\nabla{f}=a_1 \nabla{t_1}+ \dots + a_m \nabla{t_{m}}$, where $a_i$ are strictly positive real numbers;
            \item On $V$, locally at $\xi$ we have that $\hess(f)(0,\dots,0, t_{k+1}, \dots t_m)$ is positive definite in $t_{k+1}, \dots t_m$.
        \end{enumerate}
    \end{definition}

These conditions are related to standard conditions in
optimization at a point $\xi \in S$ (see
\cite{BertsekasNonlinearProgramming1999} and \cite{nie_optimality_2014}).
Notice that when BHC hold, the minimizers are non-singular, isolated points and thus finite.
It is proved in \cite{MarshallRepresentationsnonnegativepolynomials2006} that if BHC holds at every minimizer of $f$ on
    $\cS(\vb g)$ then $f-f^{*} \in \cQ(\vb g)$, which implies that the SoS
    hierarchy is exact.
    \cite{nie_optimality_2014} proved that the BHC at every minimizer of $f$, which hold generically, implies
    the SoS finite convergence property.

    In this section, we prove that, if the BHC hold, then the flat truncation property holds. For simplicity, we restrict to the case of only inequalities, i.e. we assume $\vb h = 0$. 

    \begin{theorem}\label{thm::BHC}
        Let $f \in \R[\vb{X}]$, $Q=\cQ(\vb g)$ be an Archimedean finitely
        generated quadratic module and assume that the BHC hold at every
        minimizer of $f$ on $S=\cS(\vb g)$.
        Then the SoS hierarchy is exact, the MoM hierarchy is exact, and the flat truncation holds for {all} $\lambda \in \cL_{2d}^{\min}(f;\vb g)$ when $d$ is big enough.  If conditions (i)-(iii) in \Cref{thm::flat_iff} are satisfied for the relaxation order $d$, then the flat truncation property holds. 
    \end{theorem}
    \begin{proof}
    If BHC hold at every minimizer of $f$ on
    $\cS(\vb g)$ then $S^{\min}$ is finite and$f-f^{*} \in \cQ(\vb g)$ (see \cite{MarshallRepresentationsnonnegativepolynomials2006}), which implies that the SoS
    hierarchy is exact and thus the MoM hierarchy has finite convergence. Moreover, if the BHC conditions hold at every minimizer of $f$ on $S$, then $\dim \frac{\R[\vb X]}{\supp (Q + (f-f^*))} = 0$
        (see the proof of
        \cite[Th.~2.3]{MarshallRepresentationsnonnegativepolynomials2006}, where it is shown that the field of fractions of $\R[\vb X]$ modulo any minimal prime ideal lying over $\supp (Q + (f-f^*))$ is isomorphic to $\R$, that implies $\dim \frac{\R[\vb X]}{\supp (Q + (f-f^*))} = 0$). {Then \Cref{thm::flat_iff} implies that flat truncation occurs for all $\lambda \in \cL_{2d}^{\min}(f;\vb g)$ and in particular for generic elements, when conditions (i)-(iii) are satisfied. Finally \Cref{thm::flat_implies_exact} allows to conclude that the MoM hierarchy is exact}.
    \end{proof}

        We show now that flat truncation and moment exactness hold \emph{generically} in the space $\mathcal A$ of polynomials $(f, \vb g)$ of bounded degree for which $\cQ(\vb g)$ is Archimedean.
        {We notice that if one of the $g_i$'s has even degree equal to $2d_i$, then $\mathcal A$ has nonempty interior: indeed, if we choose $g_1 = R^2-x_1^{2d_1}-\dots - x_n^{2d_1}$, for any sufficiently small perturbation $\widetilde g_1$ of $g_1$ in degree $d_1$, $\cS(\widetilde g_1)$ is compact. Therefore $Q(\widetilde{\vb g})$ satisfies the Archimedean condition for all sufficiently small perturbations $\widetilde{\vb g}$ of $\vb g$ (see also \cite[Th.~7.1.1]{marshall_positive_2008}). This is the typical situation in polynomial optimization, where it is common to set $g_1 := R^2-x_1^2-\dots - x_n^2$.}
        
        For polynomials $f \in \R[\vb X]_d$ and $g_1\in \R[\vb X]_{d_1},\dots,g_s\in
        \R[\vb X]_{d_s}$, we say that a property holds generically (or
        that the property holds for generic $f,g_{1}, \ldots, g_{s}$) if
        there exists finitely many nonzero polynomials
        $\phi_1,\dots,\phi_l$ in the coefficients of polynomials in
        $\R[\vb X]_d$ and $\R[\vb X]_{d_1},\dots,\R[\vb X]_{d_s}$ such that,
        when $\phi_1(f,\vb g)\neq 0,\dots,\phi_l(f,\vb g)\neq 0$, the
        property holds. We say that the property holds for $(f, \vb g)$ \emph{generic satisfying the Archimedean condition} if the property hold for all $(f, \vb g) \in \mathcal A \cap \{\, \phi_1(f,\vb g)\neq 0,\dots,\phi_l(f,\vb g)\neq 0 \,\}$.
    \begin{corollary}
    \label{cor::generic_exactness}
        For $f \in \R[\vb X]_d$ and $g_1\in \R[\vb X]_{d_1},\dots,g_s\in
        \R[\vb X]_{d_s}$ generic satisfying the Archimedean condition, the SoS hierarchy is exact, the MoM hierarchy is exact and the flat truncation holds for {all} $\lambda \in \cL_{2d}^{\min}(f; \vb g)$ when $d$ is big enough:
        if conditions (i)-(iii) in \Cref{thm::flat_iff} are satisfied for the relaxation order $d$, then the flat truncation holds.  
    \end{corollary}
    \begin{proof}
       By \cite[Th.~1.2]{nie_optimality_2014} BHC hold generically. We apply
       \Cref{thm::BHC} to conclude.
     \end{proof}
     {Genericity of flat truncation can also be determined by combining other previous results, as follows. In \cite{nie_certifying_2013}, it is shown that finite convergence and flat truncation are equivalent conditions under some (non-trivial) genericity assumptions (Assumption 2.1 in \cite{nie_certifying_2013}, no duality gap and $\sup = \max$ in the SoS hierarchy, see Theorem~2.2, p.~492). In \cite{nie_optimality_2014}, it is shown that BHC are generic properties, using local optimality conditions. We can combine these two results to show that under (stronger) generic conditions,  flat truncation of the Lasserre's hierarchy holds, but {some attention is needed}. First, one needs to ensure that generically there is no duality gap,
     for instance using the fact that the semialgebraic set has nonempty interior or using the results of \cite{josz_strong_2016}, and that the supremum is attained. On the other hand, \Cref{cor::generic_exactness} only requires the BHC to hold to ensure generic flat truncation.}
     
Here is an example where BHC holds.
\begin{example}[Robinson form]
    We find the minimizers of Robinson form $f = x^6+y^6+z^6+3x^2y^2z^2-x^4(y^2+z^2)-y^4(x^2+z^2)-z^4(x^2+y^2)$ on the unit sphere $h = x^2+y^2+z^2-1$. The Robinson polynomial has minimum $f^*=0$ on the unit sphere, and the minimizers on $\cV_{\R}(h)$ are:
    \[
        \frac{\sqrt{3}}{3}(\pm 1,\pm 1,\pm 1), \frac{\sqrt{2}}{2}(0,\pm 1,\pm 1), \frac{\sqrt{2}}{2}(\pm 1, 0,\pm 1), \frac{\sqrt{2}}{2}(\pm 1,\pm 1, 0).
    \]
    BHC are satisfied at every minimizer (see \cite[Ex. 3.2]{nie_optimality_2014}), flat truncation holds and we can recover the minimizers from \Cref{thm::BHC}. 
    We estimate the bounds of \Cref{thm::flat_iff} and compare with the numerical experiments. It is not necessary to check (i), since $(h) = \sqrt[\R]{\supp\cQ(\pm h)}$. {For the point (ii), we compute the interpolation degree of the minimizers. The vanishing ideal of the minimizers is
    \[
    (3xyz^3-xyz, 6z^5-5z^3+z, -6yz^4+2y^3+3yz^2-y, x^2+y^2+z^2-1, 6xz^4+2xy^2-xz^2-x, 4y^2z^2+2z^4-2y^2-3z^2+1)
    \]
    which has generators of degree $5$, and this ideal cannot be generated in degree $4$. Therefore the interpolation degree is $\intdeg=4$.
    Then $2\intdeg + 2D = 14$, and thus we expect flat truncation for $d \ge 7$. For the point (iii), we notice that, since $\deg f =6$, then $\delta \ge 3$ in \Cref{thm::flat_iff}. For the flat truncation, we then need to have $d \ge \delta + 2\intdeg + 2D - \deg(f) \ge 3 + 10 + 6 -6 -2 = 11$.}
    However, in practice for this example we have flat truncation numerically at order $6$ and not before (using the SDP solver \texttt{SDPA}). We recover a good approximation of the minimizers at this order: 
    \begin{verbatim}
        v, M = minimize(f, [h], [], X, 6)
        w, Xi = get_measure(M)
    \end{verbatim}
    Here $f^*_{\mom, 6} \approx v =-1.27211 \cdot 10^{-7}$ and
    the minimizers with positive coordinates are (all the twenty minimizers are found):
    \begin{table}[h!]
    \begin{tabularx}{\textwidth}{c|X X X X}
                & $\xi_1$& $\xi_2$& $\xi_3$& $\xi_4$\\\hline
$x$ &$0.577351068999$ & $8.812477930640~10^{-12}$ & $0.707107158043$ & $0.707107157553$ \\
$y$ &$0.577351069076$ & $0.707107158048$ & $1.271729446125~10^{-13}$ & $0.707107157555$ \\
$z$ & $0.577351066102$ & $0.707107158048$ & $0.707107158042$ & $2.478771201340~10^{-9}$
    \end{tabularx}
    \end{table}
\end{example}

\subsection{Finite semialgebraic sets}\label{sec:finite set}

In this section we consider the case when $S=\cS(\vb g)=\{\xi_1,\dots ,\xi_r\}\subset \R^{n}$ is non-empty and finite.
\begin{theorem}\label{th::finvar_exact_qua}
    Let $Q = \cQ(\vb g)$ and suppose that $\dim \frac{\R[\vb{X}]}{\supp{Q}} = 0$. Then $S$ is finite, the MoM hierarchy is exact and the flat truncation holds for all $\lambda \in \cL_{2d}^{\min}(f;\vb g)$ when $d$ is big enough. If conditions  (i)-(iii) in \Cref{thm::flat_iff} are satisfied, then flat truncation holds at the relaxation order $d$.
\end{theorem}
\begin{proof}
    {Since $\dim \frac{\R[\vb{X}]}{\supp{Q}} = 0$, we deduce that $S$ is finite and that every positive linear functional satisfies flat truncation from \Cref{thm::fin_set_qua}. Then \Cref{prop:flat truncation} implies that every truncated positive linear functional is represented by a measure, which also implies MoM finite convergence.
    We conclude applying \Cref{thm::flat_iff} and \Cref{thm::flat_implies_exact}.}
\end{proof}

As corollaries, we have that the conclusions of
\Cref{th::finvar_exact_qua} hold:
\begin{itemize}
 \item for the moment hierarchy $(\cL_{2d}(\Pi \vb g))_{d \in \N}$ (recall that we denote $\Pi \vb g$ all the products of the $g_i$'s), if $S=\cS(\vb g)=\cS(\Pi
   \vb g)$ is finite, since by the real Nullstellensatz, 
$$
\dim \frac{\R[\vb X]}{\supp Q(\Pi \vb g)} = \dim \frac{\R[\vb X]}{\supp \cO(\vb g)} =\dim \frac{\R[\vb X]}{\sqrt{\supp \cO(\vb g)}}= \dim \frac{\R[\vb X]}{\cI(\cS(\vb g))} = 0.
$$
See \cite[Th.~4.1]{nie_polynomial_2013} and \cite[Rem.~4.9]{lasserre_semidefinite_2008}. {\Cref{ex:no_convergence_zero} shows that we cannot replace the {zero-dimensionality} hypothesis in \Cref{th::finvar_exact_qua} by the finiteness of $S$ {for a general quadratic modules.}}
 \item for the moment hierarchy $(\cL_d(\vb g, \pm \vb h))_{d \in \N}$ when $\cV_{\R}(\vb h)$ is finite, since
   for $Q=Q(\vb g, \pm \vb h)$,
    \[
        \dim \frac{\R[\vb X]}{\supp Q}=\dim \frac{\R[\vb X]}{\sqrt{\supp Q}}=\dim \frac{\R[\vb X]}{\sqrt[\R]{\supp Q}} \le \dim \frac{\R[\vb X]}{\sqrt[\R]{(\vb h)}}=0.
    \]
    See \cite[Th.~1.1]{nie_polynomial_2013} and \cite{lasserre_semidefinite_2008}. This includes Polynomial Optimization problems with binary variables and equations of the form $X_i^2-X_i=0$, for which {MoM} relaxations are of particular interest, see e.g. \cite{LaurentComparisonSheraliAdamsLovaszSchrijver2003}.
\end{itemize}
{\Cref{ex::supp_non_pre} shows that \Cref{th::finvar_exact_qua} is more general than the two cases above.}

Notice that, even if the SoS hierarchy has the finite
convergence property and the MoM hierarchy is exact, it may not be SoS exact for a finite real variety, as shown in \Cref{ex::nonrad_grad} and \Cref{ex::fin_var}.

\begin{example}[{Gradient ideal, \cite{nie_minimizing_2006}}]
    We compute the minimizers of {\Cref{ex::nonrad_grad}}. Let $f=(X^4Y^2+X^2Y^4+Z^6 -2X^2Y^2Z^2)+X^8+Y^8+Z^8\in \R[X,Y,Z]$. We want to minimize $f$ over the gradient variety $\cV_{\R}\big(\pdv{f}{X},\pdv{f}{Y},\pdv{f}{Z}\big)$ with 
    $\dim \frac {\R[\vb X]} {(\pdv{f}{X},\pdv{f}{Y},\pdv{f}{Z})}=0$. By \Cref{th::finvar_exact_qua}, we deduce that flat truncation holds for an order of relaxation $d$ high enough. In this example, we have $\intdeg=0$, $d=4$, $\deg(f)=8$, $\delta \ge 4$, so that we expect flat truncation at an order $d\ge 4$, from \Cref{thm::flat_iff}.    \begin{verbatim}
        v, M = minimize(f, differentiate(f,X), [], X, 4)
        w, Xi = get_measure(M, 2.e-2)
    \end{verbatim}
    The approximation of the minimum $f^*=0$ is $v=-1.6279\cdot
    10^{-9}$, and the decomposition with a threshold of $2\cdot
    10^{-2}$, used to determine the correct numerical rank of the moment matrix, gives the following numerical approximation of the minimizer (the origin):
    \[
        \xi = (2.976731510689691~10^{-17}; -9.515032317137384~10^{-19}; 3.763401209219283~10^{-18}).
    \]
\end{example}

Generalizations of this approach {have} been investigated to make the {hierarchies}
exact, adding equality constraints satisfied by the minimizers (and
independent of the minimum $f^{*}$) to a Polynomial Optimization
Program.

As we saw in the previous example, for global optimization we can consider the gradient equations (see
\cite{nie_minimizing_2006}): obviously $\grad{f}(x^*)=\vb 0$ for all
the minimizers $x^*$ of $f$ on $S=\R^{n}$. For constrained optimization we can consider
Karush–Kuhn–Tucker (KKT) constraints, adding new variables (see
\cite{demmel_representations_2007}) or projecting them to the
variables $\vb X$ (Jacobian equations, see \cite{nie_exact_2013}).
 These are sufficient but not necessary conditions for $x^* \in S$
 being a minimizer.
To avoid this problem we can define the \emph{augmented Jacobian ideal}, see e.g. \cite{nie_exact_2013}.
 The improvement made from the KKT constraints is to consider conditions that are also necessary for being a minimizer, in the spirit of Fritz John Optimality Conditions (see \cite[Sec.~3.3.5]{BertsekasNonlinearProgramming1999}).
 
In \cite{nie_minimizing_2006},
\cite{demmel_representations_2007} and \cite{nie_exact_2013}, smoothness conditions or
radicality assumptions on the associated complex varieties are made in order to prove
finite convergence and SoS exactness.  In particular, Assumption 2.2 in \cite[Th.~2.3]{nie_exact_2013} requires the varieties defined by the active constraints to be non-singular to conclude finite convergence of the hierarchy. 
Our conditions for finite convergence and flat truncation in \Cref{th::finvar_exact_qua} are of a different nature, since they are on the finiteness of the semialgebraic set. For instance we can apply \Cref{th::finvar_exact_qua} in \Cref{exp:singular}, but Assumption 2.2 in \cite{nie_exact_2013} is not satisfied, since the minimizer is a singular point. Moreover notice that in our theorem we use only the defining inequalities $\vb g$ and not their products $\Pi \vb g$, as done in \cite[Th.~2.3]{nie_exact_2013} (in other words, we consider the quadratic module and not the preordering).

\begin{example}[Singular minimizer]\label{exp:singular}
    We minimize $f = X$ on the compact semialgebraic set
    $S=\cS(X^3-Y^2,1-X^2-Y^2)$. We have $f^*=0$ and the only minimizer is the origin, which
    is a singular point of the boundary of $S$. Thus BHC does not
    hold, and we cannot apply \Cref{thm::BHC}. 
    We have $\dim \frac{\R[\vb X]}{\supp(Q + (X))} = 0$ since $\supp(Q+(X)) \supset (X, Y^2)$, but we cannot apply \Cref{thm::flat_iff}, as we don't have finite convergence of the SoS and MoM hierarchies. Indeed $X \notin Q = \cQ(X^3-Y^2,1-X^2-Y^2)$, since $X \notin \cQ(X^3, 1-X^2)$. This implies that the SoS and MoM hierarchies do not have finite convergence, following \Cref{ex:dim2}.
    This example also shows that we cannot remove the hypothesis of MoM finite convergence in \Cref{thm::flat_iff}.
    
    To get flat truncation, we add the augmented Jacobian equations, that define a finite real variety, as we show in the following. First notice that, since $\cV(X^3-Y^2)$ is singular, Assumption 2.2 in \cite{nie_exact_2013} is not satisfied and the finite convergence of the hierarchy $\cO_{2d}(\vb g, \pm \vb{h'})$ using the augmented Jacobian equations cannot be deduced from \cite[Th.~2.3]{nie_exact_2013}. Generators for augmented Jacobian variety are $\vb{h'}=(1-X^2-Y^2)(X^3-Y^2)$, $Y(1-X^2-Y^2)$, $Y(X^3 - Y^2)$. The real roots are $(-1, 0)$, $(1, 0)$, $(0,0)$ and the two real intersections of $1-X^2-Y^2 = 0$ and $X^3-Y^2 = 0$.
    Therefore $\dim \frac{\R[\vb X]}{\supp(Q+(\vb h'))} \le \dim \frac{\R[\vb X]}{\sqrt[\R]{(\vb{h'})}} = 0$, and \Cref{thm::fin_set_qua} implies flat truncation. We  recover the minimizer considering the MoM relaxation of order $5$:
    \begin{verbatim}
        v, M = polar_minimize(f, [], [x^3-y^2,1-x^2-y^2], X, 5)
        w, Xi = get_measure(M, 2.e-3)
    \end{verbatim}
    The approximation of the minimum $f^*=0$ is $v=-0.0045$, and the
    decomposition of the moment sequence with a threshold of $2\cdot
    10^{-3}$ gives the following approximation of the minimizer (the origin):
    \[
        \xi = (-0.004514367348787526, 2.1341684460860045~ 10^{-21}).
    \]
    The error of approximation on $\xi$ is of the same order
    than the error on the minimum $f^{*}$.
\end{example}

\section{Conclusion}
We investigated the convex cones $\cL_d(\vb g)$ dual to the truncated quadratic modules $\tqgen{g}{d}$ from a new perspective. We studied the kernels of moment matrices or annihilators of moment sequences in these cones and characterize the ideal they generate (\Cref{thm::generic_t}). We focused on the zero-dimensional case and its relationships with the flat truncation property (\Cref{prop:flat truncation} and \Cref{thm::fin_set_qua}), that can be used to certify that a linear functional is {represented by} a measure.

The main contributions of the paper are the applications of the previous analysis to flat truncation in Lasserre's MoM hierarchies for Polynomial Optimization. We studied the flat truncation property in this context (\Cref{thm::flat_implies_exact}) and deduced new necessary and sufficient conditions for flat truncation (\Cref{thm::flat_iff}). These conditions can be used to show that, under regularity and thus genericity assumptions (Boundary Hessian Conditions), the flat truncation property is satisfied (\Cref{thm::BHC}, \Cref{cor::generic_exactness}). We applied these results to Polynomial Optimization on finite sets (\Cref{th::finvar_exact_qua}), generalizing and giving a unified presentation to different results in the literature.

\Cref{thm::flat_iff} provides the first known degree bounds for the flat truncation property to hold, in terms of the inequalities $\vb g$ and the objective function $f$ (in particular depending on the interpolation degree of the minimizers). An interesting question would be to investigate if it is possible to improve these degree bounds. Another possible research direction is to investigate regularity conditions, simpler than Boundary Hessian Conditions, that imply flat truncation for MoM hierarchy of a certain order $d$.

\textbf{Acknowledgments}. The authors thank D. Henrion and M. Laurent for useful discussions about the outer approximation of semialgebraic sets with moments of degree one and flat truncation properties, and the anonymous referees for their helpful suggestions for the improvement of the presentation of the paper. 

This work has been partially supported by European Union’s Horizon 2020 research and innovation programme under the Marie Skłodowska-Curie Actions, grant agreement 813211 (POEMA).

\appendix
\section{Appendix: Examples of exact and non-exact Lasserre's hierarchies}
\label{appendix}

In this appendix, we give examples showing how the notions of finite convergence and exactness of the SoS and MoM hierarchies are (and are not) related.

\paragraph{No finite convergence.}
We start presenting the first example of optimization over a finite semialgebraic set, where we do not have finite convergence of the MoM and SoS hierarchies.

\begin{example}[{see also \cite[Ex. 3.2]{scheiderer_distinguished_2005}, \cite[Rem. 3.15]{scheiderer_non-existence_2005}, \Cref{ex::scheiderer}, \Cref{ex::finite_not_exact}}]\label{ex:no_convergence_zero}
    {Consider the minimization of $f = Y-X $ on the origin $S = \cS(\vb g) =\{ \vb 0 \} \subset \R^2$, where $Q = \cQ(\vb g )=\cQ(1-X^2-Y^2, -XY, X-Y, Y-X^2)\subset \R[X,Y]$. In this case $\supp Q= \sqrt[\R]{\supp Q} = (0)$, and thus from \Cref{thm::strong_duality} there is no duality gap and $f-f^*_{\sos, d} \in \cQ_{2d}(\vb g)$ for all $d$. Then, if there is SoS finite convergence, $Y-X = f-f^* \in \cQ(\vb g)$. Since $X-Y \in \cQ(\vb g)$, finite convergence would imply that $X-Y \in \cQ(\vb g) \cap - \cQ(\vb g) = \supp \cQ(\vb g) = (0)$, a contradiction. Therefore there is no SoS finite convergence, and by strong duality there is no MoM finite convergence as well. This example shows that \cite[Th.~4.1]{nie_polynomial_2013} cannot be extended from preorderings to quadratic modules in general, but only when $\supp Q$ is a zero-dimensional ideal (see \Cref{th::finvar_exact_qua}).}
\end{example}

The next example shows that SoS and MoM hierarchies for polynomial
optimization on algebraic curves do {not necessarily have} the finite
convergence property. {For the definition of graded basis, see the paragraph before \Cref{prop:radical closed}.}
\begin{example}[\cite{scheiderer_sums_2000}]\label{ex:curvegenus1}
Let $\cC\subset \R^{n}$ be a smooth connected curve of genus $\ge 1$, with only
real points at infinity (e.g. the plane cubic defined by $Y^2 = X^3 - X$). Let $\vb h=h_{1},\ldots, h_{s}\subset \R[\vb{X}]$ be a graded basis of
$I=\cI(\cC)$ . Then there exists $f\in \R[\vb{X}]$ such
that the SoS hierarchy $\cQ_{2d}(\pm \vb{h})$ and the MoM hierarchy $\cL_{2d}(\pm \vb{h})$ have no finite convergence and
are not exact.

Indeed by \cite[Th.~3.2]{scheiderer_sums_2000}, there exists $f\in \R[\vb{X}]$ such that $f\ge 0$ on $\cC=\cS(\pm\vb{h})$, which is not a sum of squares in $\R[\cC]=\R[\vb{X}]/I$.
Consequently, $f\not \in \Sigma^{2}[\vb{X}]+I = \cQ(\pm \vb{h})$. As $f\ge 0$ on
$\cC$, its infimum $f^{*}$ is non-negative and we also have
$f-f^{*}\not \in \cQ(\pm \vb{h})$.

Using \Cref{prop:radical closed} we
deduce that $\cQ_{2d}(\pm \vb{h})$ is closed, that there is no duality gap and that the supremum
$f^{*}_{\sos,d}$ is reached.
Thus if the MoM hierarchy has finite convergence then the SoS hierarchy {also has} finite convergence and $f-f^{*} \in
\cQ_{2d}(\pm{\vb{h}})$ for some $d\in \N$. But this is in contradiction with the previous paragraph, showing that the SoS and the MoM hierarchies have no finite convergence.
\end{example}

In dimension $2$, there are also cases where the SoS and MoM hierarchies
cannot have finite convergence or be exact.
\begin{example}[\cite{marshall_positive_2008}]\label{ex:dim2} Let $g_{1}=X_{1}^{3}-X_{2}^{2},
  g_{2}=1-X_{1}$. Then $S=\cS(\vb{g})$ is a compact semialgebraic set of
  dimension $2$ and $\cO(\vb{g})$ is Archimedean.
  We have $f=X_{1}\ge 0$ on $S$ but $X_{1}\not \in \cO(\vb{g})$ (see
  \cite[Example 9.4.6(3)]{marshall_positive_2008}). The infimum of $f$
  on $S$ is $f^{*}=0$.
  Assume that we have MoM finite convergence. Using for instance \Cref{prop:radical closed}, $\cQ_{2d}(\Pi
  \vb{g})$ is closed, the supremum $f^{*}_{\sos,d}$ is reached and strong duality holds: $f^{*}_{\sos,d}=f^{*}_{\mom,d}=f^{*}=0$. Then $f-f^{*}=f\in \cO(\vb{g})$: but this
  is a contradiction. Therefore, the hierarchies $\cQ_{2d}(\Pi\vb{g}) = \cQ_{2d}(g_1, g_2 , g_1 g_2)$
  and $\cL_{2d}(\Pi\vb{g}) = \cL_{2d}(g_1, g_2 , g_1 g_2)$ cannot have finite convergence and
thus cannot be exact for $f=X_{1}$.
\end{example}
The next example shows that non-finite convergence and non-exactnesss
is always possible for semialgebraic sets of dimension $\ge 3$.
\begin{example} \label{prop::no_exactness}\label{ex:dim3}
    Let $n\ge 3$. Let $Q(\vb g)$ be an Archimedean quadratic module generated
    by $g_{1}, \ldots, g_{s}\in \R[\vb{X}]$ such that $\cS(\vb g)\subset
    \R^n$ is of dimension $m \ge 3$ and let $\vb h$ be a graded basis of $\sqrt[\R]{\supp \cQ(\vb g)}$
    (in particular $\vb h = 0$ if $m=n$, i.e. $\cS(\vb g)$ is of maximal
    dimension){. Then} there exists $f \in \R[\vb{X}]$ such that the SoS
    hierarchy $(\cQ_{2d}(\vb{g}, \pm \vb h))_{d\in \N}$ and MoM hierarchy $(\cL_{2d}(\vb{g}, \pm \vb h))_{d\in \N}$
    do not have the finite convergence property (and thus are not exact).

Indeed by \Cref{prop:radical closed} $f^*_{\sos,d}=f^*_{\mom,d}$ for $d$ big
enough and the supremum $f^*_{\sos,d}$ is reached. By \cite[Prop.~6.1]{scheiderer_sums_2000} for $m \ge 3$, $\pos(\cS(\vb g)) \supsetneq \cQ(\vb g) + (\vb h)$. So let $f \in \pos(\cS(\vb g))\setminus \cQ(\vb g)+(\vb h)$ and let $f^*$ be its minimum on $\cS(Q)$. Suppose that $f-f^*\in \cQ(\vb g)+ (\vb h)$, then $f \in \cQ(\vb g) + (\vb h)+f^*=\cQ(\vb g)+(\vb h)$, a contradiction. Then the SoS and the MoM hierarchy do not have the finite convergence property (and they are not exact).
\end{example}

\paragraph{SoS exactness, no MoM exactness.}

\begin{example}\label{ex:cylinder}
We consider in this example the unconstrained case.
    We want to find the global minimum of $f=X_{1}^2\in \R[X_{1},\ldots,X_{n}]=\R[\vb{X}]$
    for $n\ge 3$. Let $d\ge 2$, $\vb{X}'=(X_{2}, \ldots, X_{n})$  and $\overline{\lambda} \in
    \Sigma^{2}[\vb{X}']^\vee$ such that $\overline{\lambda} \not \in \cM(\R^{n-1})^{[d]}$.
    Such a linear functional exists because when $n>2$ there are non-negative
    polynomials in $\R[\vb{X}']$ which are not sum of  squares, such as the Motzkin polynomial (see \cite{reznick_concrete_1996}).
As $\Sigma^{2}[\vb{X}']$ is closed, such a
polynomial can be separated from
$\Sigma^{2}[\vb{X}']$ by a linear functional
$\overline{\lambda} \in \Sigma^{2}[\vb{X}']^\vee$, which cannot be the
truncation of a measure.
    Define $\lambda : h\mapsto \braket{\lambda}{h}=
    \braket{\overline{\lambda}}{h(0,X_{2},\ldots,X_{n})}$. We have $\lambda\in
    \Sigma^{2}[\vb{X}]^\vee$ since $\overline{\lambda} \in
    \Sigma^{2}[\vb{X}']^\vee$. Obviously
    $\braket{\lambda}{f}= 0= f^{*}$ (the minimum of
    $X_{1}^2$), $f-f^{*}=X_{1}^{2}\in \Sigma^{2}$ and the SoS hierarchy is exact. Since $\lambda$ is
    {represented by} a measure if and only if $\overline{\lambda}$ is {represented by} a measure, the MoM hierarchy cannot be exact.
\end{example}

\Cref{ex:cylinder} is an example where the number of minimizers of $f$ on $S$ is infinite. We show that non exactness can happen also when the minimizers are finite (and even when $S$ is finite!).

\begin{example}[{see also \cite[Ex. 3.2]{scheiderer_distinguished_2005}, \cite[Rem. 3.15]{scheiderer_non-existence_2005}, \Cref{ex::scheiderer}, \Cref{ex:no_convergence_zero}}]
    \label{ex::finite_not_exact}
        We want to minimize the constant function $f = 1$ on the origin $S=\cS(\vb g)=\{\vb 0\}$, where $Q = \cQ(\vb g )=\cQ(1-X^2-Y^2, -XY, X-Y, Y-X^2)\subset \R[X,Y]$. In this case $\supp Q= \sqrt[\R]{\supp Q} = (0)$. Notice that the SoS hierarchy is exact and the MoM hierarchy has finite convergence, since $f \in  \cQ_2(\vb g)$. Now suppose that the MoM hierarchy is exact, i.e. $\cL_{2d}^{\min}(f;\vb g)^{[2k]} = \cLone_{2d}(\vb g)^{[2k]} \subset \cMone(S)^{[2k]} = \{ \eval_{\vb 0}^{[2k]} \}$. Then for $\lambda^* \in \cL_{2d}(\vb g)$ we have $(\ann_k(\lambda^*)) = (\ann_k(\eval_{\vb 0})) = (X, Y)$. But from \Cref{thm::generic_t} we know that for $d, k$ big enough $(\ann_k(\lambda^*)) = \sqrt[\R]{\supp Q} = (0)$, a contradiction. Then the MoM hierarchy is not exact. Moreover the flat truncation property is not satisfied in this case: see \Cref{thm::flat_iff}.
         
        We investigate concretely this example for $d=1$. We show in \Cref{fig:1}\footnote{the variables $X, Y$ in the plots, done using \texttt{SDPA}, have been scaled by $100$ to reduce floating points errors} the plot of $\cL_2(\vb g)^{[1]}$, that is the pseudo-moments of degree one of the linear functionals in {the} dual cone of $\tqgen{g}{2}$. Notice that this is an outer approximation of $\eval_{(0,0)} \in \cL_2(\vb g)^{[1]}$ or, identifying moments of degree one with points of $\R^n$, a convex outer approximation of $S = \{ (0,0)\}$.

        One can also verify explicitly that $\cL_2(\vb g)$ has nonempty interior, as for instance $\lambda  = \lambda(\epsilon)$ defined by $\lambda_{10} = 2\epsilon$, $\lambda_{01} = \epsilon$, $\lambda_{20} = \frac{\epsilon}{2}$, $\lambda_{11} = - \epsilon^2$ and $\lambda_{02} = \frac{1}{2}$ lies in the interior of $\cL_{2}(\vb g)$ for $\epsilon > 0$ small enough.

        Notice that $\cL_2(\vb g)^{[1]} \supset \cL_3(\vb g)^{[1]} \supset \cL_4(\vb g)^{[1]} \supset \dots \supset \{ \eval_{(0,0)}^{[1]} \}$, and we have convergence since $\cQ(\vb g)$ is Archimedean. This nested outer approximations never coincide with $\{ \eval_{(0,0)}^{[1]} \}$, as we have proven before.
        \begin{figure}[h!]%
            \centering
            {{\includegraphics[width=0.49\textwidth]{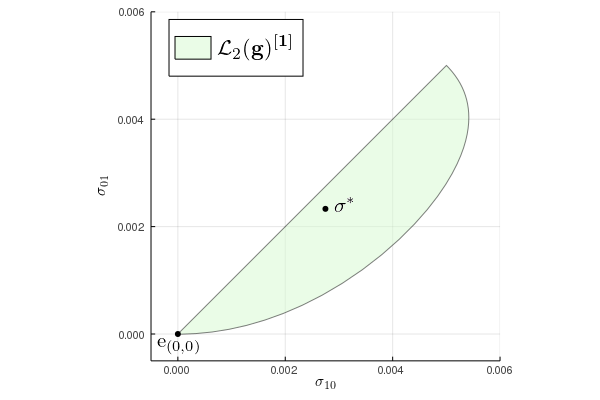} }}%
            \
            {{\includegraphics[width=0.49\textwidth]{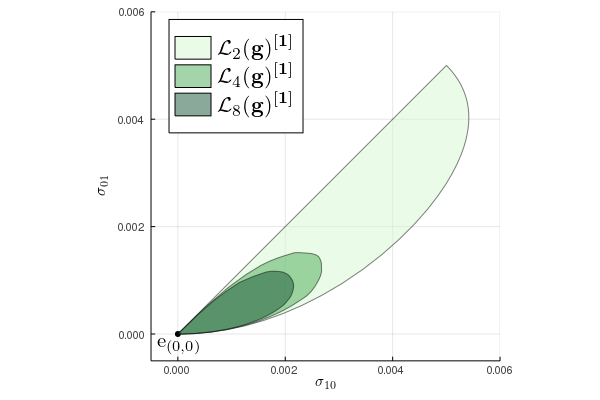} }}%
            \caption{A generic point $\lambda^*\in \cLone_2(\vb g)^{[1]}$ and moment outer approximations of $\cLone(\vb g)^{[1]} = \{\eval_{0,0}^{[1]}\}.$}%
            \label{fig:1}%
        \end{figure}
\end{example}

\paragraph{SoS finite convergence, MoM exactness.}
\begin{example}
\label{ex::nonrad_grad}
      Let $f=(X^4Y^2+X^2Y^4+Z^6 -2X^2Y^2Z^2)+X^8+Y^8+Z^8\in \R[X,Y,Z]$. We want to optimize $f$ over the gradient variety $\cV_{\R}\big(\pdv{f}{X},\pdv{f}{Y},\pdv{f}{Z}\big)$ which is {zero-dimensional} (see \cite{nie_minimizing_2006}).
      By \Cref{th::finvar_exact_qua} the flat truncation is satisfied and the MoM hierarchy is exact, and by \Cref{thm::strong_duality} and remark below the SoS has the finite convergence property (notice that $\cQ(\pm \pdv{f}{X},\pm \pdv{f}{Y},\pm \pdv{f}{Z})=\cO(\pm \pdv{f}{X},\pm \pdv{f}{Y},\pm \pdv{f}{Z})$ is Archimedean since $\cV_{\R}\big(\pdv{f}{X},\pdv{f}{Y},\pdv{f}{Z}\big)$ is compact).
      But the SoS hierarchy is not exact, as shown in \cite{nie_minimizing_2006}. 
\end{example}

\begin{example}
\label{ex::fin_var}
      Let $f=X_1$. We want to find its value at the origin, defined by $\norm{\vb{X}}^2=0$.
      As proved in \cite{nie_polynomial_2013} there is finite convergence but not exactness for the SoS hierarchy. On the other hand by \Cref{th::finvar_exact_qua} the flat truncation property is satisfied and the MoM hierarchy is exact.
\end{example}
\printbibliography
\end{document}